\documentclass[12pt,a4paper, reqno]{amsart}
\usepackage{float} 
\usepackage[utf8]{inputenc} 
\usepackage{tcolorbox} 
\usepackage{dsfont} 
\usepackage{hyperref} 
\usepackage{amsmath} 
\usepackage{amssymb} 
\usepackage{mathtools} 
\usepackage{algorithmic} 
\usepackage[ruled,vlined]{algorithm2e} 
\usepackage{amsfonts,amsthm}
\usepackage[foot]{amsaddr} 
\usepackage[margin=2.9cm]{geometry} 
\usepackage{enumitem}  

\theoremstyle{plain}
\newtheorem{Th}{Theorem}[section]
\newtheorem{lemma}[Th]{Lemma}

\newtheorem{proposition}[Th]{Proposition}

\theoremstyle{definition}
\newtheorem{definition}[Th]{Definition}

\newtheorem{?}[Th]{Problem}

\DeclareMathOperator*{\argmin}{arg\,min}
\newcommand{\supp}{\textup{supp }}

\usepackage{soul}
\newcommand{\alex}[1]{\textcolor{black}{#1}}
\newcommand{\leave}[1]{}

\usepackage{fancyhdr}
\providecommand{\keywords}[1]
{
  \small	
  \textbf{\textit{Keywords}} 
}

\begin{document}

\title{An asymptotic analysis of separating pointlike and $C^{\beta}-$curvelike singularities}
\author{Van Tiep Do$^{* \dagger}$} 
    \author{Alex Goe{\ss}mann$^{*}$}



\address{ $^{*}$ Department of Mathematics, Technische Universit{\"a}t Berlin, 10623 Berlin, Germany}

\address{ $^{\dagger}$ Vietnam National University, 334 Nguyen Trai, Thanh Xuan, Hanoi}

\email{tiepdv@math.tu-berlin.de, goessmann@tu-berlin.de}
\thanks{}


\dedicatory{}

\maketitle

\begin{abstract}
In this paper, we  present a theoretical analysis of separating images consisting of pointlike and $C^{ \beta}$-curvelike structures, where $\beta \in (1,2] $.  Our \leave{analyzing method}\alex{approach} is based on $l_1$-minimization, in which  the sparsity of the desired solution is exploited by two sparse representation systems. It is well known that for such components wavelets provide an optimally sparse representation for point singularities, whereas $\alpha$-shearlet type with $\alpha$=$\frac{2}{\beta}$ might be best adapted to the $C^{\beta}$-curvilinear singularities.
In our analysis, we first propose a reconstruction framework with theoretical guarantee on
convergence, which  is extended to use general frames instead of Parseval frames. We then construct a dual pair of bandlimited $\alpha$-shearlets which possess\alex{es a} good time and frequency localization.  Finally, we  apply the \leave{abstract} result to derive an asymptotic accuracy of the reconstructions.
In addition, we show that it is possible to separate these two components as long as $\alpha <2$, i.e.,  bandlimited $\alpha$-shearlets which range from wavelet to shearlet type do not coincide with wavelets in the sense of isotropic fashion.
\end{abstract}
\section{Introduction}
In the era of data analytics, the task of \emph{geometric separation} is of interest for various applications. For instance, astronomers might want to extract stars from filaments, or neurologists might want to separate neurons from dendrites. 
This task arises due to the fact that image data are often the superposition of different geometric components. 
In fact, numerous publications have been contributed to this area  both in the mathematical and engineering communities \cite{23,26,31,40,44,45}. 

To turn image separation into a uniquely solvable problem, the first steps are assumptions onto the shape of the components to be recovered.
Recently, the methodology of compressed sensing allows the efficient reconstruction of sparse or approximately sparse data from highly incomplete linear measurements by $l^1$-minimization or thresholding \cite{29,30,35}. 
The key ingredient is to choose two appropriate dictionaries, each one sparsely represent\alex{ing} the corresponding component but \leave{unable}\alex{failing} to sparsely represent the other. 
In recent decades, there have been various applications of compressed sensing techniques, including deblurring and deconvolution \cite{8,36}, image inpainting \cite{15,32,39}, data compression \cite{33,34}, as well as geometric separation \cite{2,7,12,20,28}. Along the way, the task of separating pointlike and curvelike structures was first introduced in \cite{12} with a theoretical recovery guarantee
by using wavelet and curvelet Parseval frames.
Indeed, wavelets are well adapted
\alex{for}\leave{to dealing with} pointlike phenomena, whereas curvelets provide optimal\leave{ly sparse} representation for images with edges. However, the limitation is that curvelets use rotation which ignores the discrete lattice structures. Later, shearlet systems, originally introduced in \cite{11} and then followed by other types \cite{9,19,37,39},  make use of shearing instead of rotation. They have been shown to share similar optimal approximation behavior with curvelets, but they allow \alex{the} unified treatment of the continuum and digital realm leading to faithful implementations compared with other known sparse representation systems like wavelets \cite{1}, ridgelets \cite{4}, curvelets \cite{5}, bandlets \cite{14} \alex{and} contourlets \cite{27}.

In addition, \leave{there appeared }a class of shearlets \alex{appeared} \leave{which used }\alex{using }flexible scaling to adapt the system according to the smoothness of the data. Among them,  universal shearlets \cite{39} form  Parseval frames by \leave{using }changing scaling parameters at each scale which are best adapted to only high frequency part, whereas compactly supported $\alpha$-shearlets \cite{42,43} with fixed scaling parameter provide superior localization but fail to form  Parseval frames. Thus, although Parseval frames play a crucial role both in applications and theoretical studies, this property is not always \leave{archived}\alex{achieved}. Overall, each of them has its own advantages and disadvantages 
depending on the approach chosen in applications.

In this paper, we consider \leave{a}\alex{the} problem of separating pointlike singularities and $C^{\beta}$-curvilinear singularities, \alex{where }$\beta \in (1,2]$. For our analysis, we construct a dual pair of bandlimited $\alpha$-shearlets
 with flexible scaling which adaptively match their decompositions to the smoothness of observed data, i.e., they provide optimal sparse representation of $C^{\beta}$-curvilinear singularities, $\beta=2/\alpha$. 
They range from wavelets ($\alpha=1$) to shearlets ($\alpha=2$). This problem is more \leave{general}\alex{involved} than \alex{when restricting to}\leave{which used} Parseval frame pair\alex{s} of wavelets-curvelets \cite{12}, or wavelets-shearlets \cite{2,20}. In our analysis, we provide a theoretical guarantee for geometric separation using general frames \leave{instead of}\alex{not restricting to} Parseval frames. 
\alex{Here we also answer the question, if separation is still possible in case of wavelets, i.e., $\alpha \rightarrow 2$.}

\subsection{Our contributions}
Our contributions in this paper \leave{reside}\alex{consist} in three main points. First, we present a theoretical guarantee for \alex{the} problem of separating two geometric components using  two general frames (Theorem \ref{1206-5}). Second, in Subsection \ref{2506-1} we construct a pair of  bandlimited $\alpha$-shearlet dual frames which provide an optimal sparse representation for $C^{\beta}$-curvilinear singularities. 
Finally, we derive an asymptotic geometric separation result for separating pointlike and $C^{\beta}$-curvelike singularities (Theorem \ref{1006-18}). Also, we show that the proposed algorithm successfully reconstructs
sub-images by $l^1$-minimization if the bandlimited $\alpha$-shearlets do not coincide with wavelets in sense of isotropic fashion. 
  
\section{Formulation of the problem} \label{1406-2}
\subsection{Notation and basic definitions}

We first start with  some basic notions and definitions.

For $f, F \in L^1(\mathbb{R}^2)$ we define the \emph{Fourier transform} and \emph{inverse Fourier transform}  by 
$$ \hat{f}(\xi)= \int_{\mathbb{R}^2} f(x) e^{-2\pi i  x^T \xi} dx,$$
$$ \check{F}(x) = \int_{\mathbb{R}^2} F(\xi) e^{2\pi i \xi^T x} d\xi,$$
with the usual extension
to  $L^2(\mathbb{R}^2)$. 

A \emph{frame} for a separable Hilbert space $\mathcal{H}$ is a countable family $\Phi = \{ \phi_i \}_{i \in I}$ in $\mathcal{H}$ for which there exist constants $0 < A \leq B < +\infty$  such that 
\begin{equation} \label{Eq1}
     A \| f \|_2^2 \leq \sum_{i \in I} | \langle f, \phi_i \rangle |^2 \leq B \|f \|_2^2, \quad  \forall f \in \mathcal{H}. 
\end{equation}
The constants $A$ and $B$ are called the \emph{lower} and \emph{upper frame bounds}, respectively. If A = B, this system is called an A-tight frame. In addition, if $A = B = 1$, then it is called a \emph{Parseval frame}.

Slightly abusing notion, we use $\Phi$ again to denote the \emph{synthesis operator}
$$ \Phi : l_2(I) \rightarrow \mathcal{H}, \quad \Phi(\{ c_i \}_{i \in I} ) = \sum_{i \in I }c_i \phi_i.$$
We denote by $\Phi^*$  the \emph{analysis operator}
$$ \Phi^* : \mathcal{H}  \rightarrow l_2(I), \quad \Phi^*(f ) = (\langle f, \phi_i \rangle)_{i \in I} .$$
The \emph{frame operator} $\mathbb{S}$ associated with a frame $\Phi = \{ \phi_i \}_{i \in I}$ is defined by
$$\mathbb{S}=\Phi \Phi^*: \mathcal{H} \rightarrow \mathcal{H}, \quad \mathbb{S} f= \sum_{i \in I} \langle f, \phi_i \rangle \phi_i. $$

Given a frame $\Phi = \{ \phi_i 
\}_{i \in I}$  for $\mathcal{H}$, there exists a sequence $\Phi^d = \{ \phi^d_i \}_{i \in I}$ in $\mathcal{H}$ such that 
\begin{equation} \label{1306-2}
   f=\Phi (\Phi^d)^* f= \sum_{i \in I}\langle f, \phi^d_i \rangle \phi_i, \; \forall f \in \mathcal{H},
\end{equation}
and 
\begin{equation} \label{1306-3}
    f=(\Phi^d) \Phi^* f= \sum_{i \in I}\langle f, \phi_i \rangle \phi_i^d,  \; \forall f \in \mathcal{H}. 
\end{equation}
If both equations hold, $\Phi^d$ is called an \emph{alternative dual
frame} (or simply a \emph{dual frame}) of $\Phi$.
Respectively, we call $\Phi^d$
an \emph{analysis pseudo-dual} of the frame $\Phi$, if only \eqref{1306-2} holds and a \emph{synthesis pseudo-dual} in case only \eqref{1306-3} holds. 
In addition, \textcolor{black}{the inequality \eqref{Eq1} implies} that $\mathbb{S}$ is a
self-adjoint, invertible operator on $\mathcal{H}$ \cite{3}.
This leads to a special dual frame 
$\{ \mathbb{S}^{-1} \phi_i \}_{i \in I}$ 
called the \emph{canonical
dual frame} of $\Phi = \{ \phi_i \}_{i \in I}$  with frame bounds $(B^{-1}, A^{-1})$. In our analysis, we denote by $\mathbb{D}_{\Phi}$ the set of all synthesis pseudo-dual of the frame $\Phi$, i.e.,
\begin{equation} \label{Eq2}
    \mathbb{D}_{\Phi}=\Big  \{ \Phi^d = \{ \phi^d_i \}_{i \in I} \in \mathcal{H} \mid f=\Phi^d \Phi^* f= \sum_{i \in I}\langle f, \phi_i \rangle \phi_i^d, \;\forall f \in \mathcal{H} \Big \}.
\end{equation} 
Obviously, $\mathbb{D}_{\Phi} \neq \emptyset$ since 
for each frame $\Phi$ the frame
    \[ \mathbb{S}^{-1}\Phi = \{ \mathbb{S}^{-1} \phi_i\}_{i \in I} \]
is dual to $\Phi$ and therefore in $\mathbb{D}_{\Phi}$.
Consequently, if $\Phi $ is a Parseval frame we have $\Phi \in \mathbb{D}_\Phi.$ For more details on frame theory we refer to \cite{3,18,22}.

\subsection{Outline}
\textcolor{black}{The rest of the paper is organized as follows.} First we formulate the general image separation problem in Section \ref{1406-2}. Next, in Section \ref{1406-3} we provide a theoretical machinery which guarantees the success of  Algorithm 1 based on notions of joint concentration and cluster coherence of two general frames. We then present the model of pointlike and curvelike singularities and 
construct radial wavelets as well as a pair of dual bandlimited $\alpha$-shearlets in Section \ref{1406-4}. 
In Section \ref{1406-5} we finally present an asymptotic separation result 
of the proposed component models by $l^1-$minimization. We close with a conclusion and outlook to further applications in Section \ref{2906-1}.

\subsection{General component separation}
\leave{An issue in image processing  is that of geometric separation.}
Given an image $f$, we assume that $f$ \leave{are}\alex{can} be composed of two geometric components, i.e., 
\begin{equation} \label{1606-10}
    f= \mathcal{P}+ \mathcal{C},
\end{equation} 
where $\mathcal{P}, \mathcal{C}$ are two unknown components which we want to recover.
\leave{It seems to be unsolvable to separate such a underdetermined system without further information since the unknowns are twice as many as knowns.}
\alex{Since the unknowns are twice as many as the equations, the task of component separation is ill-posed without additional assumptions.}
However, we often have more information about \leave{purported}\alex{the} components $\mathcal{P}, \mathcal{C}$. 
Compressed sensing techniques enable us to exactly recover these components which are sparse in  appropriate dictionaries. Here, we assume that $\mathcal{P}$ is
 smooth away from point discontinuities and 
$\mathcal{C}$ is smooth away from curvilinear singularities. 
In our analysis, we use microlocal analysis to give a heuristic understanding of why separation might be possible.
The core ingredient is based on the idea  that important coefficients are clustered geometrically in phase space.

\subsection{Recovery via $l_1$-minimization }
We consider the following  algorithm \leave{which was }proposed in the past which used Parseval frames for separating two components. We extend it to the case of two general frames $\{ \Phi_1 \}_{i \in I}, \{ \Phi_2 \}_{j \in J} $ with frame bounds $(A_1, B_1)$ and $(A_2, B_2)$ respectively.

\vskip 1mm
\begin{algorithm}[H]
\SetAlgoLined
\KwData{observed image $f$, two frames $\{ \Phi_1 \}_{i \in I},  \{ \Phi_2 \}_{j \in J}$.}

\textbf{Compute:} $(\mathcal{P}^{\star}, \mathcal{C}^{\star}),$ where 
\begin{eqnarray}
    (\mathcal{P}^{\star}, \mathcal{C}^{\star}) &= & \argmin_{f_1, f_2} \| \Phi_1^* f_1 \|_1 + \| \Phi_2^* f_2 \|_1, 
     \quad \text{subject to} \quad f_1 + f_2 = f. \qquad \label{961}
   \end{eqnarray}  \vskip 0.5mm
\KwResult{ recovered components. $\mathcal{P}^{\star}, \mathcal{C}^{\star}$.} 
\vskip 1mm
 \caption{\label{AL1-chap3} Image separation}
\end{algorithm}
\vskip 1mm

We would like to remark that here we minimize the $l^1$ norm of \alex{the} analysis coefficients when expanding the components in two frames to exploit their geometric features underlying the image. 
\alex{The success of the algorithm \eqref{961} is proven later under}
\leave{It is proven later that we ensure the success of the algorithm \eqref{961} under}
prior information that each geometric component is \leave{most }captured by the corresponding frame.

\section{Theoretical guarantee for component separation}  \label{1406-3}
\subsection{Joint concentration analysis} The notion of \emph{joint concentration} was first introduced in \cite{12}, which used Parseval frames to propose an analyzing tool for deriving the theoretical guarantee. 
There, the  joint concentration associated with two Parseval frames $\Phi_1, \Phi_2$ and sets of indexes $\Lambda_1, \Lambda_2$ is defined by
$$ \kappa(\Lambda_1, \Lambda_2)= \sup_{f \in \mathcal{H}} \frac{\| \mathds{1}_{\Lambda_1}\Phi_1^* f \|_1 + \| \mathds{1}_{\Lambda_2} \Phi_2^* f\|_1  }{\| \Phi_1^* f \|_1 + \|\Phi_2^* f\|_1}.$$
 For an extension, we modify the joint concentration associated with two frames instead of Parseval frames based on the idea  that each  frame can sparsely represent each component but can not sparsely represent the other. This philosophy plays a central role in
 the success of the proposed algorithm.

\begin{definition}  Let $\Phi_1, \Phi_2$ be two frames. We define the \emph{joint concentration } $\bar{\kappa} = \bar{\kappa}(\Lambda_1, \Lambda_2)$ with respect to sets of coefficients $\Lambda_1, \Lambda_2$ by
$$ \bar{\kappa}(\Lambda_1, \Lambda_2)= \sup_{\mathcal{P}, \mathcal{C} \in \mathcal{H}} \frac{\| \mathds{1}_{\Lambda_1}\Phi_1^* \mathcal{P} \|_1 + \| \mathds{1}_{\Lambda_2} \Phi_2^* \mathcal{C}\|_1  }{\| \Phi_1^* \mathcal{C} \|_1 + \|\Phi_2^* \mathcal{P}\|_1}.$$ 
\end{definition}

\begin{definition}
 Fix $\delta > 0$ . Given a Hilbert space $\mathcal{H}$ with a frame $\Phi, f \in \mathcal{H} $ is \emph{$\delta-$relatively sparse} in $\Phi$ with respect to $\Lambda$ if $\| \mathds{1}_{\Lambda^c} \Phi^* f \|_1 \leq \delta , $ where $\Lambda^c $ denotes $X \setminus \Lambda.$
\end{definition} 

Under the assumption of joint concentration and $\delta$-relative sparsity of the components $\mathcal{P}$ and $\mathcal{C}$, we can guarantee the success of \eqref{961}, as the next proposition shows.

\begin{proposition} \label{1006-20}
 Let $\Phi_1, \Phi_2$ be two frames with frame bounds $(A_1, B_1), (A_2, B_2),$ respectively. For $\delta_1, \delta_2 > 0,$ we assume that  $f= \mathcal{P} +  \mathcal{C} $ where $\mathcal{P}, \mathcal{C}$ is $\delta_1, \delta_2$-relatively sparse in $\Phi_1$ and $\Phi_2$ with respect to $\Lambda_1$ and $\Lambda_2$, respectively. Let $(\mathcal{P}^\star, \mathcal{C}^\star)$ solve \eqref{961} and we have $\bar{\kappa} (\Lambda_1, \Lambda_2) < \frac{1}{2}$, then  
\begin{equation} 
\| \mathcal{P}^\star - \mathcal{P} \|_2 + \| \mathcal{C}^\star - \mathcal{C} \|_2 \leq \frac{2 \max\{B_1, B_2 \} (\delta_1+ 
\delta_2)}{1 - 2 \bar{\kappa}(\Lambda_1, \Lambda_2)}.
\end{equation}  
\end{proposition}
\begin{proof}

For convenience, we set
$\bar{\kappa}:= \bar{\kappa}(\Lambda_1, \Lambda_2),  \delta = \delta_1 + \delta_2,$ and  $err:= \mathcal{P}^\star - \mathcal{P} = \mathcal{C} - \mathcal{C}^\star.$ Note here that we have $\mathcal{P}^\star + \mathcal{C}^{\star} = \mathcal{P} + \mathcal{C}=f.$
By the upper frame bounds of $\Phi_1, \Phi_2$, we obtain
\begin{eqnarray} 
\|  \mathcal{P}^\star - \mathcal{P} \|_2 + \| \mathcal{C}^\star - \mathcal{C} \|_2 & \leq & \max\{B_1, B_2 \} \Big ( \| \Phi_1^*(\mathcal{P}^\star - \mathcal{P} ) \|_2 + \| \Phi_2^*(\mathcal{C}^\star - \mathcal{C} ) \|_2 \Big ) \nonumber \\
& \leq &  \max\{B_1, B_2 \} \Big ( \| \Phi_1^*(err) \|_1 + \| \Phi_2^*(err) \|_1 \Big ). \label{ine:first}
\end{eqnarray}
Thus, we have
\begin{eqnarray*}
&& S := \| \Phi_1^*(err) \|_1 + \| \Phi_2^*(err) \|_1 \\
&\leq & \| \mathds{1}_{\Lambda_1} \Phi_1^* (err) \|_1 +\| \mathds{1}_{\Lambda_2} \Phi_2^* (err) \|_1 + \| \mathds{1}_{\Lambda_1^c} \Phi_1^* (\mathcal{P}^\star - \mathcal{P})\|_1 + \| \mathds{1}_{\Lambda_2^c} \Phi_2^* (\mathcal{C}^\star - \mathcal{C}) \|_1 \\
& \leq & \bar{\kappa} S + \| \mathds{1}_{\Lambda_1^c} \Phi_1^* \mathcal{P}^\star \|_1 +\| \mathds{1}_{\Lambda_2^c} \Phi_2^* \mathcal{C}^{\star} \|_1 + \| \mathds{1}_{\Lambda_1^c} \Phi_1^* \mathcal{P} \|_1 +\| \mathds{1}_{\Lambda_2^c} \Phi_2^* \mathcal{C}\|_1 \\
& \leq & \bar{\kappa} S + \| \mathds{1}_{\Lambda_1^c} \Phi_1^* \mathcal{P}^\star \|_1 + \| \mathds{1}_{\Lambda_2^c} \Phi_2^* \mathcal{C}^\star \|_1 +\delta_1 + \delta_2 \\
 & = & \bar{\kappa} S + \delta +  \| \Phi_1^* \mathcal{P}^\star \|_1 + \| \Phi_2^* \mathcal{C}^\star \|_1  - \| \mathds{1}_{\Lambda_1} \Phi_1^* \mathcal{P}^\star \|_1 - \| \mathds{1}_{\Lambda_2} \Phi_2^* \mathcal{C}^\star \|_1.
\end{eqnarray*}
We now exploit that $(\mathcal{P}^\star, \mathcal{C}^\star) $ is a minimizer of \eqref{961}. Therefore, we obtain
\begin{eqnarray*}
\| \Phi_1^* \mathcal{P}^\star \|_1 + \| \Phi_2^* \mathcal{C}^\star \|_1  &\leq & \| \Phi_1^* \mathcal{P} \|_1 + \| \Phi_2^* \mathcal{C} \|_1 .
\end{eqnarray*}
Thus, the triangle inequality yields
\begin{eqnarray*}
S &\leq & \bar{\kappa} S + \delta + \| \Phi_1^* \mathcal{P} \|_1 + \| \Phi_2^* \mathcal{C} \|_1 -\| \mathds{1}_{\Lambda_1} \Phi_1^* \mathcal{P}^\star \|_1 - \| \mathds{1}_{\Lambda_2} \Phi_2^* \mathcal{C}^\star \|_1  \\
&\leq & \bar{\kappa} S + \delta + \| \Phi_1^* \mathcal{P} \|_1 + \| \Phi_2^* \mathcal{C} \|_1 + \|\mathds{1}_{\Lambda_1} \Phi_1^* err\|_1 + \|\mathds{1}_{\Lambda_2} \Phi_2^* err\|_1-\| \mathds{1}_{\Lambda_1} \Phi_1^* \mathcal{P} \|_1  
 \\
 && - \| \mathds{1}_{\Lambda_2} \Phi_2^* \mathcal{C} \|_1  \\
&\leq & \bar{\kappa} S + 2 \delta + \bar{\kappa} S = 2\bar{\kappa} S + 2 \delta.
\end{eqnarray*} 
This implies   $$ S \leq \frac{2 \delta}{1 - 2 \bar{\kappa}},$$
and the claim follows in combination with \eqref{ine:first}.

\end{proof}

\subsection{Theoretical guarantee via Cluster coherence}

Typically, the information on each component is  encoded by a particular choice of the clusters $\Lambda_1, \Lambda_2$ through the number of non-zero coefficients.  By choosing such clusters we might get the successful recovery by Proposition \ref{1006-20}, but it seems hard to achieve the joint concentration. 
The notion of \emph{cluster coherence} was first introduced in \cite{12} in an attempt to transfer the theoretical guarantee based on the notion of joint concentration to another analyzing tool that enables us to check in practice. They used Parseval frames in their definition. In our paper, we modify this ansatz and  extend it to the case of generic frames instead of Parseval frames.
\begin{definition} 
Given two frames $\Phi_1 = (\Phi_{1i})_{i \in I}  $ and $\Phi_2 = (\Phi_{2j})_{j \in  J} $ with frame bounds $(A_1, B_1), (A_2, B_2)$, the \emph{cluster coherence} $\mu_c(\Lambda, \Phi_1; \Phi_2) $ of $\Phi_1$ and $\Phi_2$ with respect to the index set $\Lambda \subset I $ is defined by
$$\mu_c(\Lambda, \Phi_1; \Phi_2)=  \max_{j\in J} \sum_{i \in \Lambda} | \langle \phi_{1i}, \phi_{2j} \rangle |.$$
\end{definition}

Formulated in this way, the notion of cluster coherence encodes the geometric difference between the components in a way that can be checked in practice.
The following lemma allows to bound the joint concentration from above by the cluster coherence.
\begin{lemma} \label{1206-6}
We have
$$ \bar{\kappa}_1(\Lambda_1, \Lambda_2) \leq \inf_{\Phi_1^{d} \in \mathbb{D}_{\Phi_1}, \Phi_2^{d} \in \mathbb{D}_{\Phi_2}} \max \{ \mu_c(\Lambda_1, \Phi_1; \Phi_2^d), \mu_c(\Lambda_2, \Phi_2; \Phi_1^d)\} .$$
\end{lemma}
\begin{proof}
For $\mathcal{P}, \mathcal{C} \in \mathcal{H}, \Phi_1^{d} \in \mathbb{D}_{\Phi_1}, \Phi_2^{d} \in \mathbb{D}_{\Phi_2},$  we have $\mathcal{P}= \sum_{j \in J}\langle \mathcal{P}, \phi_{2j} \rangle \phi_{2j}^d, \; \mathcal{C}= \sum_{i \in I}\langle \mathcal{C}, \phi_{1i} \rangle \phi_{1i}^d.$ In the other words, $\mathcal{P}= \Phi_2^d \Phi_2^* \mathcal{P}, \; \mathcal{C}=\Phi_1^d \Phi_1^* \mathcal{C}$. Now we set $\alpha_1=\Phi_1^* \mathcal{C}, \; \alpha_2= \Phi_2^* \mathcal{P}, $ we then obtain $\mathcal{P}=\Phi_2^d \alpha_2,\; \mathcal{C}=\Phi_1^d \alpha_1.$
Therefore, we have
\begin{eqnarray*}
&& \| \mathds{1}_{\Lambda_1} \Phi_1^* \mathcal{P}\|_1 + \| \mathds{1}_{\Lambda_2} \Phi_2^* \mathcal{C}\|_1 
 =\| \mathds{1}_{\Lambda_1} \Phi_1^* \Phi_2^d \alpha_2 \|_1 + \| \mathds{1}_{\Lambda_2} \Phi_2^* \Phi_1^d \alpha_1 \|_1 \\
& \leq &  \sum_{i \in \Lambda_1} \Big( \sum_j | \langle \Phi_{1i}, \Phi^d_{2j} \rangle | |\alpha_{2j}| \Big) +  \sum_{j \in \Lambda_2} \Big( \sum_i | \langle  \Phi_{2j}, \Phi^d_{1i} \rangle | |\alpha_{1i}| \Big) \\
&  = &  \sum_{j} \Big( \sum_{i \in \Lambda_1} | \langle \Phi_{1i}, \Phi^d_{2j} \rangle | \Big) |\alpha_{2j}|  +  \sum_{i} \Big( \sum_{j \in \Lambda_2}| \langle \Phi_{2j}, \Phi^d_{1i} \rangle | \Big) |\alpha_{1i}| \\
& \leq & \mu_c(\Lambda_1, \Phi_1; \Phi^d_2) \| \alpha_2 \|_1 + \mu_c(\Lambda_2, \Phi_2; \Phi^d_1) \| \alpha_1 \|_1 \\
& \leq & \max \{ \mu_c(\Lambda_1, \Phi_1; \Phi_2^d), \mu_c(\Lambda_2, \Phi_2; \Phi_1^d)\} (\| \alpha_2 \|_1 + \| \alpha_1 \|_1 ) \\
& =  & \max \{ \mu_c(\Lambda_1, \Phi_1; \Phi_2^d), \mu_c(\Lambda_2, \Phi_2; \Phi_1^d)\} (\|\Phi_2^* \mathcal{P} \|_1 + \| \Phi_1^* \mathcal{C} \|_1 ).
\end{eqnarray*}
Thus, we obtain $\bar{\kappa}_1(\Lambda_1, \Lambda_2) \leq \max \{ \mu_c(\Lambda_1, \Phi_1; \Phi^d_2), \mu_c(\Lambda_2, \Phi_2; \Phi^d_1)\},$  $\forall \Phi_1^{d} \in \mathbb{D}_{\Phi_1}, \forall \Phi_2^{d} \in \mathbb{D}_{\Phi_2}.$ This completes the proof.
\end{proof}

We can now present our theoretical guarantee for the procedure \eqref{961} to be convergent using two generic frames instead of Parseval frames.

\begin{Th} \label{1206-5}
 Let $\Phi_1, \Phi_2$ be two frames with frame bounds $(A_1, B_1), (A_2, B_2),$ respectively. For $\delta_1, \delta_2 > 0,$ we suppose that $f \in \mathcal{H}$ can be decomposed as $f= \mathcal{P} +  \mathcal{C} $ so that each component $\mathcal{P}, \mathcal{C}$ is $\delta_1, \delta_2-$relatively sparse in $\Phi_1$ and $\Phi_2$ with respect to $\Lambda_1$ and $\Lambda_2$, respectively. Let $(\mathcal{P}^\star, \mathcal{C}^\star)$ solve \eqref{961}. If we have $\mu_c(\Lambda_1, \Lambda_2):=\inf\limits_{\Phi_1^{d} \in \mathbb{D}_{\Phi_1}, \Phi_2^{d} \in \mathbb{D}_{\Phi_2}}  \max \{ \mu_c(\Lambda_1, \Phi_1; \Phi^d_2), \mu_c(\Lambda_2, \Phi_2; \Phi^d_1)\} < \frac{1}{2}$, then  
\begin{equation}  \label{1306-1}
\| \mathcal{P}^\star - \mathcal{P} \|_2 + \| \mathcal{C}^\star - \mathcal{C} \|_2 \leq \frac{2 \max\{B_1, B_2 \} (\delta_1+ 
\delta_2)}{1 - 2 \mu_c(\Lambda_1, \Lambda_2)}.
\end{equation}  
\end{Th}
\begin{proof}
The proof follows directly from Proposition \ref{1006-20} and Lemma \ref{1206-6}.
\end{proof}

We would like to remark that in case $\Phi_1, \Phi_2$ are Parseval frames we have  $\mu_c(\Lambda_1, \Lambda_2) \leq \max \{ \mu_c(\Lambda_1, \Phi_1; \Phi_2), \mu_c(\Lambda_2, \Phi_2; \Phi_1)\}$ since $\Phi_1 \in \mathbb{D}_{\Phi_1}, \Phi_2 \in \mathbb{D}_{\Phi_2}.$ This consequence is exactly the result used in several papers \cite{2,10,12,20,23} which chose Parseval frames as sparsifying systems. Thus, our theoretical guarantee based on the notion of cluster coherence \eqref{1306-1} is  more general than are shown in aforementioned papers since we can choose an other dual instead of using itself. \textcolor{black}{Although it is not easy to construct a dual, this theoretical guarantee  actually shows that we can use synthesis pseudo-dual's properties instead of its explicit construction. The nice property we need in our analysis is that  dual frame elements have good time-frequency localization which leads to small cluster coherence.} 

\textcolor{black}{For \alex{the} construction of a shearlet dual, some \alex{works} \cite{21,24} have shown \leave{in  which  possess }desirable properties such as well \leave{localized}\alex{localization} and highly directional sensitivity.  In \alex{the} next section, we use the approach from \cite{21} to construct a dual 
pair of bandlimited $\alpha$-shearlets for our analysis.}
\section{Component separation } \label{1406-4}
\subsection{Mathematical model of components}

Consider the image separation problem \eqref{1606-10}, we now introduce the model of components. 
In our analysis, we assume that the pointlike part $\mathcal{P}$ is modeled as 
\begin{equation} \label{EQ60-chap3}
    \mathcal{P}(x)= \sum_{1}^K c_i | x- x_i |^{-\lambda_i},
\end{equation}
where two sequences of constants $\{ \lambda_i \}_{i=1}^K, \{c_i\}_{i=1}^K$ satisfy $0< \lambda_i <2, 0 < c_i, \; \forall i=1,2,\dots, K $.
We choose $\lambda_i <2$ to bound the energy of the component in frequency domain as the scale goes finer. The choice of $\{ \lambda_i \}_{i=1}^K$  depends on each problem of image separation which make components comparable.

For the curvilinear singularities, we first recall the \emph{Schwartz functions} or the \emph{rapidly decreasing functions}
\begin{equation*} \label{1606-1}
     \mathcal{S}(\mathbb{R}^2)= \Big \{ f \in C^\infty(\mathbb{R}^2)  \mid   \forall K, N \in \mathbb{N}_0,\sup_{x \in \mathbb{R}^2} (1 + |x|^2 )^{-N/2} \sum_{|\alpha |\leq K} | D^\alpha f(x) | < \infty \Big \}. 
\end{equation*}
Let $\sigma: [0,1] \rightarrow \mathbb{R}^2$ be a closed $C^{\beta}$ curve, $\beta \in (1,2],$ with non vanishing curvature everywhere. 
We first consider the model of curvilinear singularity $\mathcal{C}$ as
\begin{equation}
    \mathcal{C}=\int_{\mathbb{R}^2} \delta_{\sigma(t)} dt, 
\end{equation}
where $\delta_x$ denotes the usual Dirac delta distribution located at $x$.
It is well known that the class of $\alpha$-shearlets using $\alpha$-scaling can sparsely represent such curvilinear structures \cite{42,43}. 
For the sake of simplicity, we restrict our model to the case of \emph{line singularity}. The reader should be aware of the fact that by Tubular neighborhood theorem we can extend it to the general case. Intuitively, we can use the technique in \cite{12} to first partition the curve $\sigma$ into small pieces and then bend them to the form of a line singularity, see \cite[Section 6]{12} for details.

Similarly as introduced in several papers \cite{15,39}, we \leave{consider }model \leave{of }the line distribution $w \mathcal{L}$ acting on Schwartz functions by
\begin{equation} \label{EQ61-chap3}
\langle w \mathcal{L} , f \rangle = \int_{\rho}^{\rho}  w(x_1) f(x_1, 0) dx_1, \; f \in \mathcal{S}(\mathbb{R}^2),
\end{equation} 
where $w$ is a weighted function such that $0 \not\equiv w \in C^{\infty}(\mathbb{R}),\supp w \subset [-\rho, \rho ],$ for some $\rho >0$, and $0 \leq w(x) \leq 1, \forall x \in [-\rho, \rho]. $ 
 For such a setting, we now approach the question if it possible to separate point-singularities and the line singularities using wavelets and bandlimited $\alpha$-shearlet types which interpolate from wavelet type ($\alpha=2$) to shearlet type ($\alpha=1$). In our analysis, we prove that we can separate them as long as bandlimited $\alpha$-shearlets do not coincide with the wavelet type, i.e., $2 >\alpha  \geq 1.$
 
Among well-known systems, wavelets  provide an optimal sparse representation to the
pointwise singularities, whereas shearlets are efficient for curvilinear structures.
In what follows, we introduce the construction of these two sparsifying systems.
\subsection{Wavelet frames}
To sparsely present $\mathcal{P}$, we choose radial wavelets which form a Parseval frame with perfectly isotropic generating elements.
We modify the construction of radial wavelets as follows.

Let $\Xi$ be a Schwartz function on $\mathbb{R}^2$ such that  $\supp \hat{\Xi} \subset [-\frac{1}{16}, \frac{1}{16}]$, $0 \leq \hat{\Xi}(\theta) \leq 1$ for $\theta \in\mathbb{R}$ and $\hat{\Xi}(\theta) =1$ for $\theta \in [-\frac{1}{32}, \frac{1}{32}]$. We now define the \emph{low-pass function} $\Omega(\xi)$ and the \emph{window function} $W(\xi)$ for $j \in \mathbb{N} $ and $\xi= (\xi_1,\xi_2 ) \in \mathbb{R}^2,$
\begin{equation}
     \hat{\Omega}(\xi):= \hat{\Xi}(\xi_1) \hat{\Xi}(\xi_2), \label{2108-5}
\end{equation}
\begin{equation} 
    W(\xi):= \sqrt{\hat{\Omega}^2(2^{-2}\xi)-\hat{\Omega}^2(\xi)}, \quad W_j(\xi):= W(2^{-2j} \xi). \label{11061} 
\end{equation} 
By definition, $\sup W \subset [-\frac{1}{4}, \frac{1}{4}]^2 \setminus [-\frac{1}{32}, \frac{1}{32}]^2 $ and $W_j$ is compactly supported in the corona
\begin{equation} \label{Pr2_EQ11}
 \mathcal{A}_j := [-2^{2j-2}, 2^{2j-2}]^2 \setminus [-2^{2j-5}, 2^{2j-5}]^2 , \quad  \forall j \geq 1.
\end{equation}
In addition, we obtain the partition of unity property
\begin{equation} \label{Pr2_CT3}
 \hat{\Omega}^2(\xi) + \sum_{j \geq 0 } W_j^2(\xi) = 1, \quad \forall \xi \in \mathbb{R}^2.
\end{equation}
The radial wavelets $\mathbf{\Psi} = \{ \psi_{j,m}\}_{(j,m)} \cup \{\Omega(\cdot-m) \}, j \in \mathbb{N}_0, m \in \mathbb{Z}^2 $ which form a Pareval frame are then defined by their Fourier transforms
$$ \hat{\psi}_{j,m}(\xi)=2^{-2j}W_j(\xi)e^{2 \pi i\xi^Tm/2^{2j}},  j \in \mathbb{N}_0, m \in \mathbb{Z}^2.$$
Since the low frequency part is not of interest to us in our analysis, we therefore simply write $\mathbf{\Psi}= \{ \psi_{j,m} \}_{(j,m)}$ at some points.
\subsection{A pair of bandlimited $\alpha$-shearlet dual  frames} \label{2506-1}
This section is devoted to the construction of a pair of \emph{bandlimited $\alpha$-shearlet} dual frames which possess many desirable properties. We choose shearlets as it is widely accepted that shearlets in general provide optimal sparse representation for images which \text{are} governed by curvilinear structures \cite{6,13,2}. 
Motivated by \cite{21}, we modify the construction of the shearlet frame pair and extend it to the case of $\alpha$-scaling instead of parabolic scaling since
$\alpha$-shearlet type with $\alpha=\frac{2}{\beta}$, $\alpha \in [1,2)$, might be best adapted to $C^{\beta}$ curvilinear singularities. 
In addition, we modify the Fourier domain decomposition to make it comparable to the wavelet frames.

We first define
the  \emph{scaling} and \emph{shearing matrix}   by
\begin{equation} \label{EQ33-chap3}
A_{\alpha, \rm{h}} := 
\begin{bmatrix}
    2^2   & 0 \\
      0  & 2^\alpha \\
   \end{bmatrix}, \quad S_{\rm{h}} := \begin{bmatrix}
    1  & 1 \\
      0  & 1 \\
   \end{bmatrix} ,
\end{equation}

\begin{equation} \label{EQ34-chap3}
 A_{\alpha, \rm{v}} := 
\begin{bmatrix}
    2^\alpha & 0 \\
      0  & 2^2 \\
   \end{bmatrix}, \quad S_{\rm{v}} := \begin{bmatrix}
    1  & 0 \\
      1  & 1 \\
   \end{bmatrix},
\end{equation}
where $\alpha \in [1, 2)$ is the \emph{scaling parameter}. 
Let $v \in C^\infty (\mathbb{R})$ be a \emph{bump function} such that $\supp v \subset [-\frac{3}{2},\frac{3}{2}] $ and 
\begin{equation*}
    \sum_{l=-2}^2| v(\omega -l)|^2  =1, \text{ for }  \omega \in [-\frac{3}{2},\frac{3}{2}].
\end{equation*}
Consequently, the following holds \alex{for }$\leave{\forall }j \geq 0,  \omega \in [-\frac{3}{2},\frac{3}{2}] ,$ 
\begin{equation}
   \sum_{l=-\lceil 2\cdot 2^{(2-\alpha )j} \rceil}^{\lceil  2\cdot 2^{(2-\alpha )j} \rceil} |v(2^{(2-\alpha)j} \omega -l)|^2=1. \label{866}
\end{equation}

Next, we define the  \emph{cone functions} $V_1, V_2$ by
 \begin{equation} \label{Pr2_EQ32}
 V_{\rm{h}}(\xi): = v\Big ( \frac{\xi_2}{\xi_1} \Big ), \quad V_{\rm{v}}(\xi): = v\Big ( \frac{\xi_1}{\xi_2} \Big ). 
 \end{equation}
\emph{horizontal frequency cone} and the \emph{vertical frequency cone}
\begin{equation} 
\mathcal{C}_{\rm{h}} := \Big \{ (\xi_1, \xi_2) \in \mathbb{R}^2 :  |\xi_1| \geq \frac{1}{8}, \Big |\frac{\xi_2}{\xi_1} \Big | \leq \frac{3}{2}  \Big \},
\end{equation}
\begin{equation} 
     \mathcal{C}_{\rm{v}} := \Big \{ (\xi_1, \xi_2) \in \mathbb{R}^2 :|\xi_2| \geq \frac{1}{8}, \Big  |\frac{\xi_1}{\xi_2} \Big | \leq \frac{3}{2}  \Big \},
\end{equation}
and low frequency part
\begin{equation} \label{0407-1}
     \mathcal{C}_{\rm{0}} := \Big \{ (\xi_1, \xi_2) \in \mathbb{R}^2 : |\xi_1|, |\xi_2| \leq 1\}.
\end{equation}

A system of shearlets is then defined by
$$ \mathbf{\Phi}= \Big \{  \phi^{\alpha, \iota}_{j,l,k} (x), \iota= \{\rm{h, v} \}, j \in \mathbb{N}, - \lceil 2 \cdot 2^{(2-\alpha )j} \rceil \leq l \leq  \lceil 2 \cdot 2^{(2-\alpha )j} \rceil, k \in \mathbb{Z}^2   \Big \},$$
where $\hat{\phi}^{\alpha, \iota}_{j,l,k}(\xi)=W_j(\xi) V_\iota \Big ( \xi^T A_{\alpha, \iota}^{-j} S^{-l}_\iota \Big ) e^{2\pi i \xi^T A_{\alpha, \iota}^{-j} S^{-l}_{\iota} k}, \iota=\{\rm{h,v} \}.$
For an illustration,  Fig. \ref{2108-2} shows the tiling of the frequency domain induced by shearlets.

By the definition $\hat{\phi}_{j,l,k}^{\alpha, \rm{\iota}} $ has compact support in the trapezoidal region
\begin{equation}
     \supp \hat{\phi}_{j,l,k}^{\alpha, \rm{v}}=\Big \{  \xi \in \mathbb{R}^2 \ | \ \xi_2 \in [-2^{2j-2}, 2^{2j-2}] \setminus [-2^{2j-5}, 2^{2j-5}], \Big | \frac{\xi_1}{\xi_2} - l2^{-(2-\alpha )j} \Big | \Big \}, \label{96100}
 \end{equation}
  \begin{equation}
    \supp \hat{\phi}_{j,l,k}^{\alpha, \rm{h}}=\Big \{  \xi \in \mathbb{R}^2 \ | \ \xi_1 \in [-2^{2j-2}, 2^{2j-2}] \setminus [-2^{2j-5}, 2^{2j-5}], \Big | \frac{\xi_2}{\xi_1} - l2^{-(2-\alpha )j} \Big | \Big \}. \label{96101}
 \end{equation}
For convenience, we also define two following compact sets   
\begin{equation}
\mathcal{C}^{*}_{\rm{h}} :=\Big \{ \xi \in \mathbb{R}^2: \frac{1}{32} \leq |\xi_1| \leq \frac{1}{4}, \Big | \frac{\xi_2}{\xi_1} \Big | \leq \frac{3}{2}  \Big \}  \supset \supp \hat{\phi}_{j,l,k}^{\alpha, \rm{h}}(\xi^T S^{l}_{\rm{h}}A^{j}_{\alpha, \rm{h}} ), \label{867}
\end{equation}

\begin{equation}
    \mathcal{C}^{*}_{\rm{v}} :=\Big \{ \xi \in \mathbb{R}^2: \frac{1}{32} \leq |\xi_2| \leq \frac{1}{4}, \Big | \frac{\xi_1}{\xi_2} \Big | \leq \frac{3}{2}  \Big \}  \supset \supp \hat{\phi}_{j,l,k}^{\alpha, \rm{v}}(\xi^T S^{l}_{\rm{v}}A^{j}_{\alpha, \rm{v}} ). \label{868}
\end{equation}

\begin{figure}[H] 
\centering
\includegraphics[width=220pt, height=170pt]{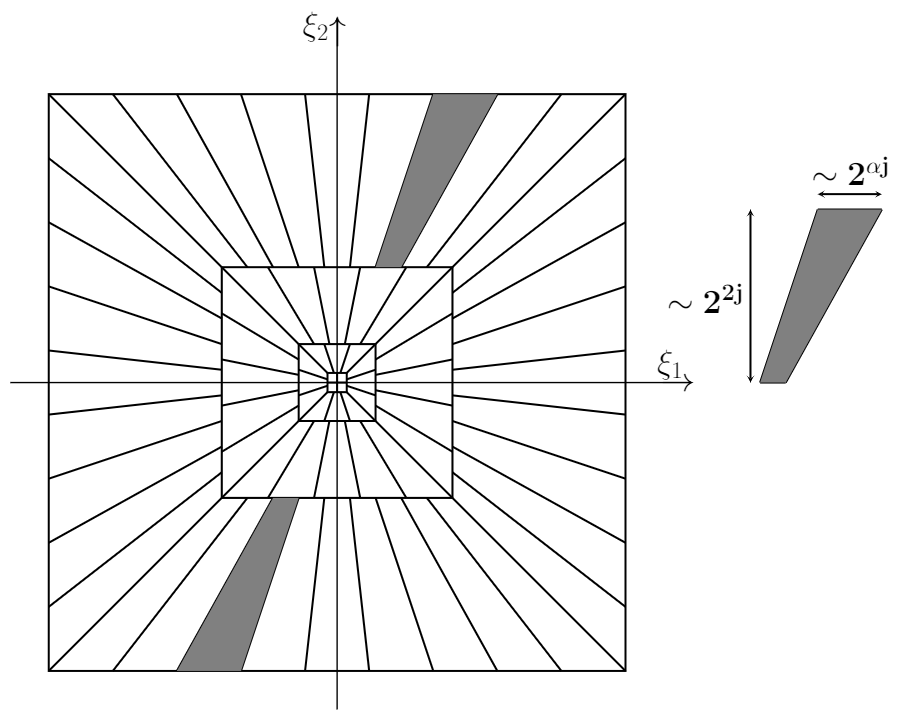}
\caption{\label{2108-2} Frequency tiling of shearlets. } 
\end{figure} 

The key ingredient for the construction of a pair of dual shearlet frames is derived from the following Parseval frames for sub-spaces $L^2(\mathcal{C}_{\rm{h}})^\vee, L^2(\mathcal{C}_{\rm{v}})^{\vee}$ which can be then aggregated to  form a frame for $L^2(\mathbb{R}^2).$

\begin{lemma}
The system $\mathbf{\Phi^{\iota}_\alpha}=\{ \phi^{\alpha, \iota}_{j,l,k}, j \in \mathbb{N}_0,  -  \lceil 2 \cdot 2^{(2-\alpha )j} \rceil \leq l \leq   \lceil 2 \cdot 2^{(2-\alpha )j} \rceil,$ $ k  \in \mathbb{Z}^2  \}$ forms a Parseval frame for $L_2(\mathcal{C}_\iota)^\vee, \iota =\{ \rm{h, v} \}  $ where
$L_2(\mathcal{C}_\iota)^\vee= \{ f \in L_2(\mathbb{R}^2): \supp \hat{f} \subset \mathcal{C}_\iota\}. $
\end{lemma}
\begin{proof}
We \alex{only}\leave{first} consider $\iota=\rm{h}, $ \alex{since} the other case is done similarly. Indeed, by \eqref{Pr2_CT3} and \eqref{866} we have
\begin{eqnarray}
\sum_{j \in \mathbb{N}_0} \sum_{l=-\lceil  2^{(2-\alpha )j+1} \rceil}^{\lceil  2^{(2-\alpha )j+1} \rceil}  |\hat{\phi}_{j,l,k}^{\alpha, \rm{h}}(\xi)  |^2  &= & \sum_{j \in \mathbb{N}_0} |W_j(\xi)|^2  \sum_{l=-\lceil  2^{(2-\alpha )j+1} \rceil}^{\lceil  2^{(2-\alpha )j+1} \rceil} |v(2^{(2-\alpha)j}\frac{\xi_2}{\xi_1}-l)|^2\nonumber \\
& =& 1, \quad \forall \xi \in \mathcal{C}_{\rm{h}}. \nonumber
\end{eqnarray}
Here we note that low-pass function $\Omega(\xi)=0,\; \forall \xi \in \mathcal{C}_{\mathrm{h}}.$
By using Parseval’s identity and the observation that $\supp \hat{\phi}_{j,l,k}^{\alpha,\rm{h}}(\xi^T S^l_{\rm{h}}A^j_{\alpha, \rm{h}}) \subset \mathcal{C}_{\rm{h}}^* \subset [-\frac{1}{2}, \frac{1}{2}]^2, $ we concludes the claim by standard arguments.
\end{proof}

We will construct a pair of dual shearlet frames by carefully patching together three Parseval frames $\mathbf{\Phi_{\alpha}^\iota}, \iota=\{ \rm{0,h,v}\}$, where the construction of translation-invariant Parseval frame $\mathbf{\Phi^0}$ for $L^2(\mathcal{C}_0)$ is well-known. We first define corresponding cones of the dual 
\begin{equation} 
\mathcal{C}^d_{\rm{h}} := \Big \{ (\xi_1, \xi_2) \in \mathbb{R}^2 :  |\xi_1| \geq \frac{1}{4}, \Big |\frac{\xi_2}{\xi_1} \Big | \leq \frac{4}{3}  \Big \},
\end{equation}
\begin{equation} 
     \mathcal{C}^d_{\rm{v}} := \Big \{ (\xi_1, \xi_2) \in \mathbb{R}^2 :|\xi_2| \geq \frac{1}{4}, \Big  |\frac{\xi_1}{\xi_2} \Big | \leq \frac{4}{3}  \Big \},
\end{equation}
and low frequency part
\begin{equation} 
     \mathcal{C}^d_{\rm{0}} := \Big \{ (\xi_1, \xi_2) \in \mathbb{R}^2 : |\xi_1|, |\xi_2| \leq \frac{2}{3} \}.
\end{equation}
For an illustration, we refer to Figure \ref{1706-1}.
\begin{figure}[H] 
  \centering
  \includegraphics[width=0.5\textwidth]{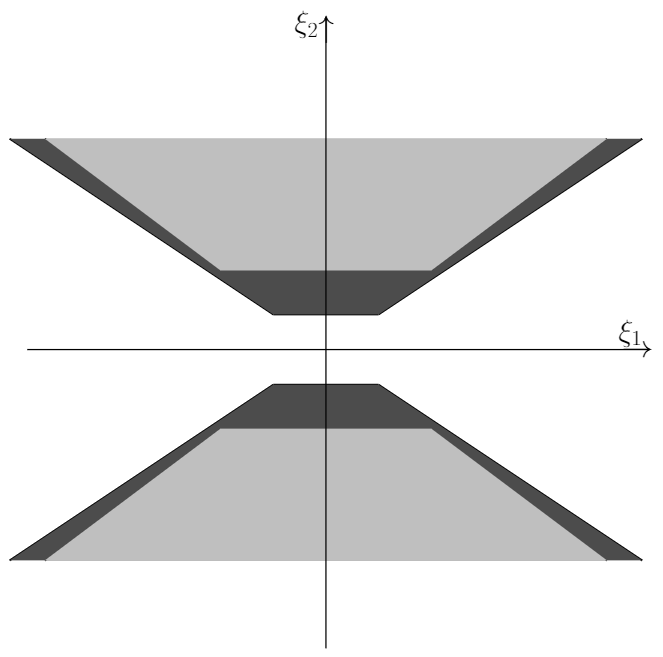}
  \caption{Cones $\mathcal{C}_{\rm{v}}$ (black) and $\mathcal{C}_{\rm{v}}^{d}$ (gray) with $\mathcal{C}_{\rm{v}}^{d}\subset \mathcal{C}_{\rm{v}} $ in frequency domain. \label{1706-1}}
\end{figure}

Next, we choose  $\chi_\iota, \gamma_\iota$ in $C^{\infty}(\mathbb{R}^2),  \iota = \rm{0,h,v}$ by the following lemma.
\begin{lemma} \label{865}
There exist $\chi_\iota, \gamma_\iota$ in $C^{\infty}(\mathbb{R}^2),  \iota = \rm{0,h,v}$ such that  the following properties hold\alex{:}
\begin{enumerate}[label=(\roman*)]
    \item $\supp \chi_{\rm{\iota}}, \gamma_{\rm{\iota}}\subset \mathcal{C}_{\rm{\iota}}, \; \iota =\{\rm{0,h,v} \},$  
     $\sum_{\iota \in \{ \rm{0,h,v}\} } \chi_{\rm{\iota}}(\xi)=1, \; \forall \xi \in \mathbb{R}^2. $ \\
    \item $\chi_{\rm{\iota}}(\xi^TS^{l}_{\iota}A^{j}_{\alpha, \iota}), \gamma_{\rm{\iota}}(\xi^TS^{l}_{\iota}A^{j}_{\alpha, \iota}) \in C^{\infty}(\mathcal{C}^*_\iota) $ with norms 
    $$\|\chi_{\rm{\iota}}((\cdot)^T S^{l}_{\iota}A^{j}_{\alpha, \iota}) \|_{C^N(\mathcal{C}^*_\iota) }, \; \|\gamma_{\rm{\iota}}((\cdot)^TS^{l}_{\iota}A^{j}_{\alpha, \iota}) \|_{C^N(\mathcal{C}^*_\iota) } \leq C_N,$$  where constants $C_N $ are independent of $j$. 
\end{enumerate}
\end{lemma}

\begin{proof}
Let $\chi_0 \in C^{\infty}, \supp \chi =\mathcal{C}_0=\{ \xi \in \mathbb{R}^2: |\xi| \leq 1 \}$ and $\chi_0(\xi)=1, \forall \xi=(\xi_1, \xi_2): |\xi|\leq \frac{2}{3}.  $ We now define
$$ \chi_{\rm{h}}:= g_{\rm{h}}(\xi_1)h_{\rm{h}}(\frac{\xi_2}{\xi_1}), \quad \chi_{\rm{v}}(\xi_1, \xi_2):=\chi_{\rm{h}}(\xi_2, \xi_1)=g_h(\xi_2)h_{\rm{h}}(\frac{\xi_1}{\xi_2}),$$
where $g_{\rm{h}}, h_{\rm{h}}$ are real-valued functions such that
$ g_{\rm{h}}, h_{\rm{h}} \in C^{\infty}(\mathbb{R}),\supp g_{\rm{h}} = [\frac{1}{8}, \infty),  g_{\rm{h}}|_{[\frac{1}{2},\infty)} \equiv 1, $ and $ \supp h_{\rm{h}} \subset [-\frac{3}{2}, \frac{3}{2}], h_{\rm{h}}|_{[-\frac{4}{3}, \frac{4}{3}]} \equiv 1.$

Next, we define
$$ \gamma_{\rm{h}}(\xi):=g_{\rm{v}}(\xi_1)h_{\rm{v}}(\frac{\xi_2}{\xi_1}), \quad \gamma_{\rm{v}}(\xi)=g_{\rm{v}}(\xi_2) \Big [ 1-h_{\rm{v}}(\frac{\xi_2}{\xi_1})  \Big ], $$
where $\supp g_{\rm{v}} = [\frac{1}{4}, \infty), g_{\rm{v}}|_{[\frac{1}{2}, \infty)}\equiv 1,$ and $\supp h_{\rm{v}}=[-\frac{4}{3}, \frac{4}{3}], h_{\rm{v}}|_{[-\frac{3}{4}, \frac{3}{4}]}\equiv 1.$ We now choose  $\gamma_0(\xi)=1-\chi_{\rm{h}}(\xi) \gamma_{\rm{h}}(\xi)-\chi_{\rm{v}}(\xi)\gamma_{\rm{v}}(\xi), \forall \xi \in \mathbb{R}^2.$ We prove that $\chi_\iota, \gamma_\iota$ satisfy desired properties. Indeed,

i)  Obviously, $\supp \chi_\iota \subset \mathcal{C}_\iota, \iota =\{\rm{0,h,v}\}.$ By definition, $\supp \gamma_{\rm{h}} \subset \mathcal{C}_{\rm{h}}^d, $ and $\supp \gamma_{\rm{v}}\subset \mathcal{C}_{\rm{v}}^d$ since $1-h_{\rm{v}}(\frac{\xi_2}{\xi_1}) =0$ for $\xi=(\xi_1, \xi_2): \Big | \frac{\xi_1}{\xi_2}\Big | > \frac{4}{3}.$

Now it remains to show that $ \supp \gamma_0 \subset \mathcal{C}_0$. For this, we need to prove $\chi_{\rm{h}}(\xi)\gamma_{\rm{h}}(\xi)+ \chi_{\rm{v}}(\xi)\gamma_{\rm{v}}(\xi)=1$  for $\xi, |\xi| > \frac{2}{3}. $  Indeed, we first observe
\begin{equation}
    h_{\rm{h}}(\frac{\xi_2}{\xi_1})\cdot h_{\rm{v}}(\frac{\xi_2}{\xi_1}) = h_{\rm{v}}(\frac{\xi_2}{\xi_1}), \; \forall \xi \in \mathbb{R}^2, \label{861}
\end{equation}
and
\begin{equation}
    h_{\rm{h}}(\frac{\xi_1}{\xi_2}) \cdot \Big [ 1-h_{\rm{v}}(\frac{\xi_2}{\xi_1}) \Big ]  = 1-h_{\rm{v}}(\frac{\xi_2}{\xi_1}), \; \forall \xi \in \mathbb{R}^2. \label{862}
\end{equation}
Next we consider four cases

Case 1:  $|\xi_1| \geq \frac{2}{3}, |\xi_2| < \frac{1}{2}.$  Since  $|\xi_1|>\frac{1}{2}, \Big |\frac{\xi_1}{\xi_2} \Big |>\frac{4}{3}, \Big |\frac{\xi_2}{\xi_1} \Big | <\frac{3}{4},$ we obtain $g_{\rm{h}}(\xi_1)=g_{\rm{v}}(\xi_1)=h_{\rm{h}}(\frac{\xi_2}{\xi_1}) =h_{\rm{v}}(\frac{\xi_2}{\xi_1})=1$. Thus, $\chi_{\rm{h}}(\xi)\gamma_{\rm{h}}(\xi)+ \chi_{\rm{v}}(\xi)\gamma_{\rm{v}}(\xi)=1.$

Case 2:  $|\xi_1| \geq \frac{2}{3}, |\xi_2| \geq \frac{1}{2}.$ We have $g_{\rm{h}}(\xi_1) =g_{\rm{v}}(\xi_1)=g_{\rm{h}}(\xi_2) =g_{\rm{v}}(\xi_2)=1.$ This implies
\begin{eqnarray}
\chi_{\rm{h}}(\xi)\gamma_{\rm{h}}(\xi)+ \chi_{\rm{v}}(\xi)\gamma_{\rm{v}}(\xi) &=& h_{\rm{h}}(\frac{\xi_2}{\xi_1}) h_{\rm{v}}(\frac{\xi_2}{\xi_1}) + h_{\rm{h}}(\frac{\xi_1}{\xi_2})\Big [ 1- h_{\rm{v}}(\frac{\xi_2}{\xi_1})\Big ]\nonumber \\
& \stackrel{\mathclap{\normalfont{ \eqref{861}+\eqref{862}}}} = & \quad \; h_{\rm{v}}(\frac{\xi_2}{\xi_1}) + 1 - h_{\rm{v}}(\frac{\xi_2}{\xi_1}) =1.\nonumber
\end{eqnarray}

Case 3:  $|\xi_2| \geq \frac{2}{3}, |\xi_1| < \frac{1}{2}.$ Since $|\xi_2| \geq \frac{1}{2}, \Big |\frac{\xi_2}{\xi_1}\Big | > \frac{4}{3}, \Big | \frac{\xi_1}{\xi_2}\Big | < \frac{3}{4},$ we obtain $g_{\rm{h}}(\xi_2) =g_{\rm{v}}(\xi_2)=1, h_{\rm{h}}(\frac{\xi_1}{\xi_2})=1, h_{\rm{v}}(\frac{\xi_2}{\xi_1})=0.$ Thus,
$\chi_{\rm{h}}(\xi)\gamma_{\rm{h}}(\xi)+ \chi_{\rm{v}}(\xi)\gamma_{\rm{v}}(\xi)=0+1=1.$

Case 4:  $|\xi_2| \geq \frac{2}{3}, |\xi_1| \geq \frac{1}{2}.$ Since  $g_{\rm{h}}(\xi_1) =g_{\rm{v}}(\xi_1)=g_{\rm{h}}(\xi_2) =g_{\rm{v}}(\xi_2)=1$, we derive
\begin{eqnarray}
\chi_{\rm{h}}(\xi)\gamma_{\rm{h}}(\xi)+ \chi_{\rm{v}}(\xi)\gamma_{\rm{v}}(\xi) &=& h_{\rm{h}}(\frac{\xi_2}{\xi_1}) h_{\rm{v}}(\frac{\xi_2}{\xi_1}) + h_{\rm{h}}(\frac{\xi_1}{\xi_2})\Big [ 1- h_{\rm{v}}(\frac{\xi_2}{\xi_1})\Big ]\nonumber \\
& \stackrel{\mathclap{\normalfont{ \eqref{861}+\eqref{862}}}} = & \qquad  h_{\rm{v}}(\frac{\xi_2}{\xi_1}) + 1 - h_{\rm{v}}(\frac{\xi_2}{\xi_1}) =1.\nonumber
\end{eqnarray}
Thus, we obtain $\supp \gamma_0 \subset \mathcal{C}_0^d.$

 On the other hand, by definition we have $\gamma_0(\xi) + \chi_{\rm{h}}(\xi)\gamma_{\rm{h}}(\xi)+ \chi_{\rm{v}}(\xi) \gamma_{\rm{v}}(\xi)=1, \forall \xi \in \mathbb{R}^2.$ In addition, $\supp \gamma_0 \subset \mathcal{C}^d_0=\{\xi \in \mathbb{R}^2: |\xi| \leq \frac{2}{3} \}$ and 
$\chi_0 \equiv 1, \forall \xi \in \mathcal{C}_0^d,$ we obtain $\chi_0(\xi) \gamma_0(\xi) = \gamma_0(\xi), \; \forall \xi \in \mathbb{R}^2.$ Thus, $\sum_{\iota \in \rm{0,h,v}} \chi_\iota(\xi) \gamma_\iota(\xi) =1. $ This concludes the claim.

iii)  First we consider  \begin{eqnarray}
\chi_{\rm{h}}(\xi^TS^{l}_{\rm{h}}A^{j}_{\alpha, \rm{h}})&=&\chi_{\rm{h}}(2^{2j}\xi_1, 2^{\alpha j }(\xi_2 + l\xi_1))  \nonumber \\
&=& g_{\rm{h}}(2^{2j}\xi_1) h_{\rm{h}}(2^{-(2-\alpha)j}(\frac{\xi_2}{\xi_1}+l)). \nonumber
\end{eqnarray}
Since  we have $g_{\rm{h}}(2^{2j}\xi_1) \equiv 1, \forall \xi \in \mathcal{C}_{\rm{h}}^*,$ for $j \geq 2.$
Obviously, $\chi_{\rm{h}}(\xi^TS^{l}_{\rm{h}}A^{j}_{\alpha, \rm{h}}) \in C^{\infty}(\mathcal{C}^*_{\rm{h}}) $ with
$\|\chi_{\rm{h}}((\cdot)^T S^{l}_{\rm{h}}A^{j}_{\alpha, \rm{h}}) \|_{C^N(\mathcal{C}^*_{\rm{h}}) } \leq C_N.$ The proof works analogously to case of 
$\chi_{\rm{v}}(\xi^TS^{l}_{\rm{v}}A^{j}_{\alpha, \rm{v}}), \gamma_{\rm{h}}(\xi^TS^{l}_{\rm{h}}A^{j}_{\alpha, \rm{h}}).$
It remains to prove for $\gamma_{\rm{v}}(\xi^TS^{l}_{\rm{v}}A^{j}_{\alpha, \rm{v}}),$ we have
\begin{eqnarray*}
\gamma_{\rm{v}}(\xi^TS^{l}_{\rm{v}}A^{j}_{\alpha, \rm{v}}) &=& g_{\rm{v}}(2^{2j}\xi_2)\Big [ 1-h_{\rm{v}}(\frac{2^{2j}\xi_2}{2^{\alpha j}(l\xi_1 + \xi_2})\Big ].
\end{eqnarray*}
The observation $\supp h_{\rm{v}}=[-\frac{4}{3},\frac{4}{3}]$  implies $h_{\rm{v}}(\frac{2^{2j}\xi_2}{2^{\alpha j}(l\xi_1 + \xi_2})=0, $ for $\Big | \frac{2^{2j}\xi_2}{2^{\alpha j}(l\xi_1 + \xi_2} \Big | > \frac{4}{3}$. In the other words, for $j \geq 2, \gamma_{\rm{v}}(\xi^TS^{l}_{\rm{v}}A^{j}_{\alpha, \rm{v}}) \neq 1$ only for $l \gtrsim 2^{(2-\alpha)j}.$ By a direct computation, for $j \geq 2, l \gtrsim 2^{(2-\alpha)j},$ we obtain
$G(\xi):=2^{(2-\alpha)j} \frac{\xi_2}{l \xi_1 + \xi_2} \in C^{\infty}(\mathcal{C}^*_{\rm{v}})$ and $\Big | \frac{\partial^N }{\partial \xi_1^N} G(\xi)\Big |, \Big | \frac{\partial^N  }{\partial \xi_2^N}G(\xi) \Big | \leq C_N$, where $C_N$ are constants independent of $j.$
This finishes the proof.
\end{proof}

Let us now go back  to the starting point of constructing a dual pair of shearlet types. 
 We define two representation systems of bandlimited $\alpha$-shearlet associated with those functions in Lemma \ref{865} by
 $$\mathbf{\Phi_\alpha} = \Big \{ \check{\chi}_\iota* \phi^{\alpha,\rm{\iota}}_{j,l,k} (x), \iota= \{\rm{h,v} \}, j \in \mathbb{N}_0, - \lceil 2 \cdot 2^{(2-\alpha )j} \rceil \leq l \leq  \lceil 2 \cdot 2^{(2-\alpha )j} \rceil, k \in \mathbb{Z}^2   \Big \} $$
 $\qquad  \bigcup \{ \check{\chi}_0*\phi^0_k(x), k \in \mathbb{Z}^2\},$ \\
and
$$ \mathbf{\Phi^{d}_{\alpha}}= \Big \{\check{\gamma}_\iota* \phi^{\alpha, \iota}_{j,l,k} (x), \iota= \{ \rm{h, v} \}, j \in \mathbb{N}_0, - \lceil 2 \cdot 2^{(2-\alpha )j} \rceil \leq l \leq  \lceil 2 \cdot 2^{(2-\alpha )j} \rceil, k \in \mathbb{Z}^2   \Big \}.$$
 $\qquad  \bigcup \{ \check{\gamma}_0*\phi^0_k(x), k \in \mathbb{Z}^2\},$ \\
where $\mathbf{\Phi^0} = \{ \phi^0_k(x)\}_{k \in \mathbb{Z}^2}$ forms a Parseval frame for $L^2(\mathcal{C}_0)$.

As mentioned before, these two systems are a natural extension of shearlets by using flexible scaling to accommodate the smoothness of the data. They consist of a
countable collection of well-local\leave{l}ized shearlet elements at various locations, scales, and orientations. In addition, we combine two Parseval frames by using $\chi_{\iota}, \gamma_{\iota}$ instead of $\gamma_{\iota}, \frac{\chi_{\iota}}{\gamma_{\iota}}$ used in \cite{21} which might lead to bad behavior when we consider higher derivatives of shearlet elements. 

For convenience, let us use the following index set of shearlets
\begin{eqnarray} 
\Delta &:=& \Big \{(j,l,k, \rm{\iota}) \mid j \geq 0, l \in \mathbb{Z}, |l| \leq \lceil 2\cdot  2^{(2-\alpha )j} \rceil , k \in \mathbb{Z}^2, \rm{\iota} \in \{ \rm{h}, \rm{v} \}  \Big \}.  \label{2008-1}
\end{eqnarray} 
\begin{lemma} The systems $\mathbf{\Phi_\alpha}, \mathbf{\Phi_\alpha^d}$ form a pair of shearlet frames  for $L^2(\mathbb{R}^2)$. Consequently, we have 
$\mathbf{\Phi_\alpha^d} \in \mathbb{D}_{\mathbf{\Phi_\alpha}}.$
\end{lemma}
\begin{proof}
First we prove that $\mathbf{\Phi_\alpha}$ forms a frame for $L^2(\mathbb{R}^2).$ Indeed, the observation $\supp \chi_\iota \in \mathcal{C}_{\iota}, \iota=\{\rm{0,h,v} \}$ implies $\supp \chi_\iota \hat{f} \in \mathcal{C}_\iota, \iota=\{ \rm{0,h,v}\}.$ 
Now we exploit that $\mathbf{\Phi^0}$ and $\mathbf{\Phi_\alpha^\iota}=\{ \phi^{\alpha,\iota}_{j,l,k}, j\in \mathbb{N}_0,  -  \lceil 2 \cdot 2^{(2-\alpha )j} \rceil \leq l \leq  \lceil 2 \cdot 2^{(2-\alpha )j} \rceil, k \in \mathbb{Z}^2  \}, \iota=\{\rm{ h,v} \}$ constitute Parseval frames for $\mathcal{C}_\iota, \iota=\{ \rm{0,h,v} \}$ \alex{and} obtain
\begin{eqnarray}
\sum_{\substack{ {\{ (j,l,k,\iota) \in \Delta} \} \\ \cup \{\iota=0, k \in \mathbb{Z}^2 \} }}| \langle  f, \Phi_\alpha \rangle |^2
&=&
\sum_{(j,l,k,\iota) \in \Delta} | \langle  \hat{f}, \chi_\iota \hat{\phi}^{\alpha, \iota}_{j,l,k} \rangle |^2 + \sum_{k \in \mathbb{Z}^2} |\langle \hat{f}, \chi_0 \hat{\phi}^0_k \rangle |^2 \nonumber \\
&=& \sum_{(j,l,k,\iota) \in \Delta} | \langle \chi_\iota \hat{f}, \hat{\phi}^{\alpha, \iota}_{j,l,k}  \rangle |^2 + \sum_{k \in \mathbb{Z}^2} |\langle \chi_0 \hat{f},  \hat{\phi}^0_k \rangle|^2 \nonumber \\
&=&\sum_{\iota \in \{ \rm{0,h,v} \}}\|\chi_\iota \hat{f} \|_2^2 \leq \sum_{\iota \in \{ 0,h,v\} } \|\chi_\iota \|_\infty^2 \cdot  \|f \|^2_2. \label{461}
\end{eqnarray}
In addition, 
\begin{eqnarray}
3\sup_{\iota \in \{ \rm{0,h,v} \}} \| \gamma_\iota \|_\infty^2 \cdot \sum_{\substack{ {\{(j,l,k,\iota) \in \Delta}\} \\ \cup \{\iota=0, k \in \mathbb{Z}^2 \} }} | \langle  f, \Phi_\alpha \rangle |^2
&=& 3\sup_{\iota \in \{ \rm{0,h,v} \}} \| \gamma_\iota \|_\infty^2 \cdot \sum_{\iota \in \{ \rm{0,h,v} \}}\|\chi_\iota \hat{f} \|_2^2  \nonumber \\
&\geq &  \| \sum_{\iota \in \{ 0,h,v\} } \chi_\iota \gamma_\iota \hat{f} \|_2^2 =\|f\|_2^2. \label{462}
\end{eqnarray}
Combine \eqref{461} and \eqref{462} we obtain
\begin{equation}
    \frac{1}{3\sup_{\iota \in \{ \rm{0,h,v} \}} \| \gamma_\iota \|_\infty^2 } \cdot \|f \|_2^2 \leq  \sum_{\substack{ {\{(j,l,k,\iota) \in \Delta} \} \\ \cup \{\iota=0, k \in \mathbb{Z}^2 \} }} | \langle  f, \Phi_\alpha \rangle |^2
\leq 
  \sum_{\iota \in \{ 0,h,v\} } \|\chi_\iota \|_\infty^2  \cdot \|f\|_2^2. \label{463}
\end{equation}
Similarly, we have
\begin{equation}
    \frac{1}{3\sup_{\iota \in \{ \rm{0,h,v} \}} \| \chi_\iota \|_\infty^2 } \cdot \|f \|_2^2 \leq  \sum_{\substack{ {\{(j,l,k,\iota) \in \Delta}\} \\ \cup \{\iota=0, k \in \mathbb{Z}^2 \} }} | \langle  f, \Phi^{d}_\alpha \rangle |^2
\leq 
  \sum_{\iota \in \{ 0,h,v\} } \|\gamma_\iota \|_\infty^2  \cdot \|f\|_2^2. \label{464}
\end{equation}
It remains to show that $\mathbf{\Phi_\alpha}, \mathbf{\Phi_\alpha^d}$ form a pair of dual frame\alex{s}. \alex{Due to} Plancherel's theorem and $\mathbf{\Phi^0}, \mathbf{\Phi_\alpha^\iota}, \iota=\{ \rm{h,v}\}$ \alex{being}\leave{are} Parseval frames for $L^2(\mathcal{C}_\iota)^\vee, \iota=\{ \rm{0,h,v}\}$, we have
\begin{eqnarray*}
\sum_{\substack{ {\{(j,l,k,\iota) \in \Delta} \}\\ \cup \{\iota=0, k \in \mathbb{Z}^2 \} }} \langle f, \Phi_\alpha \rangle \hat{\Phi}_\alpha^d &= &\sum_{(j,l,k,\iota) \in \Delta}\langle \hat{f}, \chi_\iota \hat{\phi}^{\alpha, \iota}_{j,l,k} \rangle \gamma_\iota  \hat{\phi}^{\alpha, \iota}_{j,l,k} + \sum_{k \in \mathbb{Z}^2} \langle \hat{f}, \chi_0 \phi^0_k \rangle \gamma_0 \phi^0_k \\
& = & \sum_{(j,l,k,\iota) \in \Delta} \langle \chi_\iota \gamma_\iota \hat{f},  \hat{\phi}^{\alpha, \iota}_{j,l,k} \rangle  \hat{\phi}^{\alpha, \iota}_{j,l,k}  + \sum_{k \in \mathbb{Z}^2} \langle \chi_0 \gamma_0  \hat{f},  \phi^0_k \rangle \phi^0_k \\
&=& \sum_{\iota \in \{ \rm{0, h,v} \}} \chi_\iota \gamma_\iota 
\hat{f} =\hat{f}.
\end{eqnarray*}
\alex{This finishes the proof.}
\leave{This concludes the claim.}
\end{proof}

Shearlets have been studied extensively so far due to their highly directional performance and smooth digital grid, whereas wavelets are best adapted to anisotropic features like point singularities. Bandlimited $\alpha$-shearlets  allow for a unified \alex{treatment of wavelets and shearlets}\leave{view from wavelets to shearlets}. If $\alpha=1$ we  obtain the bandlimited shearlet frame by using parabolic scaling. 
 Different from the literature \cite{21}, we here rescaled the parameter $j$ to $2j$. Therefore, the spatial footprints of shearlets are
of size $2^{-j}$ times $2^{-2j}$ instead of $2^{-j/2}$ times $2^{-j}$.
Also, if $\alpha$ approaches 2 the elements of $\alpha$-shearlets scale in an isotropic fashion. Thus, $\mathbf{\Phi_\alpha}=\{ \phi_{j,l,k}^{\alpha, \iota} \}_{(j,l,k,\iota) \in \Delta} \cup \{ \phi^0_k\}_{k \in \mathbb{Z}^2}$ can be viewed as a special \alex{instance}\leave{case} of wavelets. Consider the shearlet system $\mathbf{\Phi_{\alpha}} $ the tiling of the
frequency domain on the vertical cone $\mathcal{C}_{\rm{v}}$ is illustrated in Figure \ref{1706-2}.

\begin{figure}[H] \label{1706-2}
  \centering
  \includegraphics[width=0.55\textwidth]{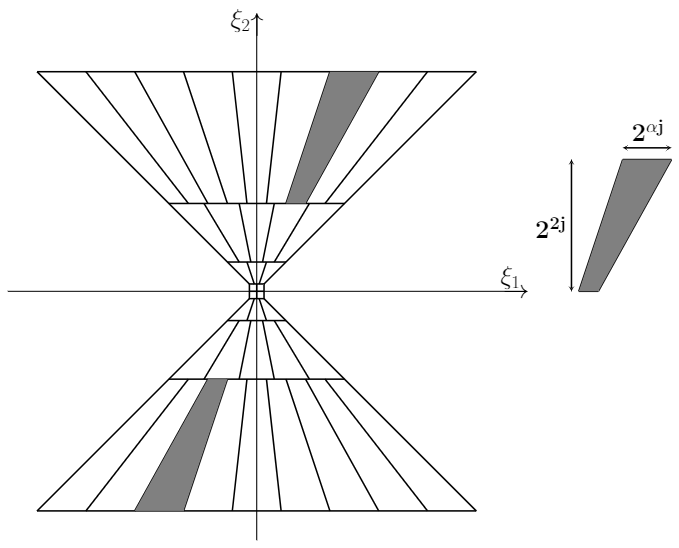}

  \caption{The frequency tilling of the vertical cone $\mathcal{C}_{\rm{v}},$ with the  support of one shearlet highlighted.}
\end{figure}

The reader should be aware of the fact that we can construct a system satisfying a milder condition \eqref{1306-3}, i.e., an element of $\mathbb{D}_{\mathbf{\Phi}}$, instead of forming a dual. In fact, if there exists a well-localized synthesis pseudo-dual the success of the proposed algorithm is guaranteed. 

\section{Multi-scale component separation} \label{1406-5}
For $\alpha \in [1, 2)$, we first fix a constant $\epsilon$ such that
\begin{equation}
    0< \epsilon < \frac{2-\alpha }{4}. \label{468}
\end{equation}
Using the window function $W_j$ by  \eqref{11061}, we define a class of \emph{frequency filters} $F_j$ defined by its Fourier transform 
\begin{equation} 
   \hat{F}_j(\xi):=W_j(\xi) = W(\xi/2^{2j}), \quad \forall j \geq 1, \xi \in \mathbb{R}^2, 
\end{equation}
and in low frequency part
$$ \hat{F}_{low}(\xi):= \Omega(\xi), \quad \xi \in \mathbb{R}^2,$$ 
where $\Omega$ is defined in \eqref{2108-5}.
Using these filters to decompose $f$ into sub-images by $f_j = F_j* f,$
the original image $f$ is then recover\alex{e}d from its pieces $f_j$ by \eqref{Pr2_CT3}
\begin{equation} \label{2606-1}
   f = F_{low}*F_{low}*f+\sum_{j \in \mathbb{N}_0} F_j * f_j.
\end{equation} 
We intend to apply Algorithm 1 for each sub-image $f_j$ which is assumed to be $\mathcal{P}_j + w \mathcal{L}_j$, where $\mathcal{P}_j = F_j* \mathcal{P}, w \mathcal{L}_j=F_j*w\mathcal{L}$. The whole signal $f$ is then reconstructed by \eqref{2606-1}.

These filters allow us to consider the pointlike and curvelike part at different scales.
The separation problem is meaningful at each scale if the components are comparable in sense of their energy.  A main issue is the balance of the energy of each component \cite{12,23}, i.e., we may assume that $ \{ \gamma_i\}_{i=1}^K, \{ c_i \}_{i=1}^K $ are chosen satisfying $\| \mathcal{P}_j \|_2  \approx \| \mathcal{C}_j \|_2$ at each scale $j$. In that case, we can choose $\max\limits_{i=1, 2, \cdots, K}\{ \lambda_i \} =\frac{3}{2}.$
However, we will \leave{prove in general case and do} not \alex{restrict to this}\leave{assume this} energy balancing condition in our analysis. We later show that the success of the proposed algorithm is guaranteed as long as the energy of components are not too small, i.e., they satisfy a  lower bound at each scale. This guarantees the performance even in case of different energies.

We would like to emphasize that  there are very few elements interacting significantly with the discontinuities, i.e., few elements of the wavelet and shearlet expansions are enough to provide accurate approximations of \alex{the} components. This allows us to define the following clusters which will be then exploited to derive small error approximation as well as the cluster coherence \alex{of}
\begin{equation}
    \Lambda_{1j}=\Big \{(j^{\prime}, m) \mid j^{\prime} \in \mathbb{Z}, |j^{\prime} -j| \leq 1, m=(m_1, m_2) \in \mathbb{Z}^2, \sqrt{m_1^2+ m_2^2} \leq 2^{\epsilon j} \Big \},
\end{equation}
\alex{and}
\begin{equation}
    \Lambda_{2j}=\Big \{(j^{\prime},l,k, \rm{v}) \in \Delta \mid j^{\prime} \in \mathbb{Z}, |j^{\prime}-j| \leq 1,  k=(k_1,k_2) \in \mathbb{Z}^2, |k_2-lk_1| \leq 2^{\epsilon j} \Big \},
\end{equation}
where $\Delta$ is defined in \eqref{2008-1}.
 \textcolor{black}{\leave{For an intuition}\alex{Intuitively}, the geometry underlying these two components should be \leave{relatively }different so that a chosen frame pair can provide sufficient small cluster sparsity and cluster coherence, which guarantee the success of Algorithm 1. Indeed, although points and curves can \leave{be }overlap\leave{ped} in spatial domain, \leave{but }their wavefront set \alex{might be}\leave{are} quite different in phase space consisting of a pair of positions and directions. Here the wavefront set of a distribution $f$ is the set of positions and orientations at which $f$ is not smooth. This difference intuitively enables us to separate the \alex{components} by microlocal analysis,  see \cite{12} for \alex{a more comprehensive discussion.}\leave{such microlocal viewpoint.}}

We now present our main separation result.
\begin{Th} \label{1006-18}
Consider a wavelet frame $\mathbf{\Psi }= \{ \psi_{j,m}\}_{(j,m)} \cup \{\Omega(\cdot-m) \}, j \in \mathbb{N}_0, m \in \mathbb{Z}^2 $ and a bandlimited $\alpha$-shearlet frame $\mathbf{\Phi_\alpha}=\{ \phi_{j,l,k}^{\alpha, \iota}\}_{(j,l,k, \iota) \in \Delta} \cup \{ \phi^0_k\}_{k \in \mathbb{Z}^2}, \alpha \in [1,2)$ . Then \leave{asymptotic }separation based on scale $j$ is achieved by \eqref{961} \alex{in the limit of large $j$}. Namely, we have
\begin{equation} 
 \frac{\| \mathcal{P}_j^\star - \mathcal{P}_j  \|_2 }{\| \mathcal{P}_j \|_2 } + \frac{\| \mathcal{C}_j^\star - w\mathcal{L}_j  \|_2 }{\| w\mathcal{L}_j \|_2 } = o(2^{-Nj}) \rightarrow 0, \quad j \rightarrow \infty,  \label{EQ201-chap3}
\end{equation}
for all $N \geq 0$, where $(\mathcal{P}_j^\star,\mathcal{C}_j^\star) $ is the solution of \eqref{961} and $(\mathcal{C}_j, \mathcal{T}_{s,j})$ are purported components. 
\end{Th}
We delay the proof until later after we have introduced our main tools in Subsections \ref{1506-1}, \ref{1306-5}. \alex{We will}\leave{In summary, we intend to} use Theorem \ref{1206-5} for the proof of this theorem. For this, we need to show that  clusters $\Lambda_{1,j}, \Lambda_{2,j}$ provide sparse representation of pointlike and curvelike singularities in \alex{the} sense that the sparsity $\delta=\delta_{1,2}+ \delta_{2,j}$ is negligible. In addition, the cluster coherence $\mu_c=\mu_c(\Lambda_{1,j}, \mathbf{\Psi}; \mathbf{\Phi^\alpha})+ \mu_c(\Lambda_{2,j}, \mathbf{\Phi_\alpha}; \mathbf{\Psi})$ is also sufficiently small. Those conditions ensure the success of Algorithm 1 as a result of Theorem \ref{1206-5}.
We would like to remark that in our analysis we need \alex{the}\leave{a} fast decay \leave{property }of the dual frame rather than its explicit construction. 

As a consequence, Theorem \ref{1006-18} implies that it is possible to separate pointlike and curvelike structures using wavelets and $\alpha$-shearlets if $\alpha$-shearlets do not share the isotropic feature with wavelets, i.e., $\alpha <2$. \textcolor{black}{Indeed, traditional wavelets which are based on isotropic dilations are not very effective when dealing with multivariate data. 
In contrast, shearlets show a greater ability to capture anisotropic features by applying a different dilation factor along the two axes. Here  the parameter $\alpha$ measures the degree of anisotropy. If $\alpha =2$ the scaling matrices $A_{\alpha, \iota}, \iota=\{ \rm{h,v}\}$ become isotropic dilations which show similar behavior as wavelets in \alex{the} sense of directional sensitivity. }

 The next two subsections are dedicated to the estimate of cluster coherence and the cluster sparsity.

\subsection{Estimate of cluster coherence} \label{1506-1}
 First, we introduce the  notation which is useful in rest of the paper
\begin{equation}
 \langle |x| \rangle = (1 + |x|^2)^{1/2}. \label{0507-1}
\end{equation} 
The following lemma shows that wavelet and shearlet elements are well-localized at various scales, positions and directions.
 
\begin{lemma} \label{465}
Consider a wavelet frame $\mathbf{\Psi}$ and a band limited $\alpha$-shearlet frame $\mathbf{\Phi_\alpha} $. Then there exists a \alex{universal} constant $C_N$ independent of $j$  such that the following estimates hold for any arbitrary integer $ N =1,2, \dots $ 
\begin{enumerate}[label=(\roman*)]
\item $ |\psi_{j,m}(x)| \leq C_N \cdot 2^{2j} \cdot  \langle |2^{2j} x_1 + m_1|\rangle^{-N}\langle |2^{2j} x_2 + m_2|\rangle^{-N}.$
\item $ | \check{\chi}_{\rm{v}}*\phi_{j, l, k}^{\alpha, \rm{v}}(x)| \leq C_N \cdot 2^{(2+\alpha)j/2} \cdot \langle |2^{\alpha j}x_1 + k_1|\rangle^{-N} \langle |2^{2j}x_2+ l2^{\alpha j } x_1+k_2|\rangle^{-N}, $ \\
$ |\check{\chi}_{\rm{h}}* \phi_{j, l, k}^{\alpha, \rm{h}}(x)| \leq C_N \cdot 2^{(2+\alpha) j/2} \cdot \langle |2^{2j}x_1+l2^{\alpha j}x_2-k_1)|\rangle^{-N} \langle |2^{\alpha j}x_2-k_2|\rangle^{-N}, $ \\
and similarly \\
$ | \check{\gamma}_{\rm{v}}*\phi_{j, l, k}^{\alpha, \rm{v}}(x)| \leq C_N \cdot 2^{(2+\alpha)j/2} \cdot \langle |2^{\alpha j}x_1 + k_1|\rangle^{-N} \langle |2^{2j}x_2+ l2^{\alpha j } x_1+k_2|\rangle^{-N}, $ \\
$ |\check{\gamma}_{\rm{h}}* \phi_{j, l, k}^{\alpha, \rm{h}}(x)| \leq C_N \cdot 2^{(2+\alpha) j/2} \cdot \langle |2^{2j}x_1+l2^{\alpha j}x_2-k_1)|\rangle^{-N} \langle |2^{\alpha j}x_2-k_2|\rangle^{-N}. $ 
\item  $ |\langle \psi_{j,m}, \check{\chi}_\iota*\phi^{\alpha, \iota}_{j, l, k} \rangle|  \leq C_N \cdot 2^{-(2-\alpha)j/2},$ \quad   $\forall 
 \rm{\iota \in \{ \rm{v}, \rm{h} \}},$ \\
  $ |\langle \psi_{j,m}, \check{\gamma}_\iota*\phi^{\alpha, \iota}_{j, l, k} \rangle|  \leq C_N \cdot 2^{-(2-\alpha)j/2},$ \quad   $\forall 
 \rm{\iota \in \{ \rm{v}, \rm{h} \}}.$ 
\end{enumerate}
\end{lemma}

\begin{proof}
i) With the variable $\zeta =2^{-2j}\xi $, we have 
\begin{eqnarray*}
|\psi_{j,m}(x)| & =& \Big | \int_{\mathbb{R}^2} 2^{-2j}W_j(\xi)e^{2 \pi i\xi^Tm/2^{2j}} e^{2 \pi i \xi^T x} d\xi\Big |  \\
&=&  \Big | \int_{\mathbb{R}^2} 2^{2j}W(\zeta) e^{2\pi i \zeta^T (2^{2j}x+m)} d\zeta \Big |. \\
\end{eqnarray*}
By integration by parts for $N_1, N_2=1,2, \dots, $  with respect to $\zeta_1, \zeta_2$, respectively, we have
\begin{eqnarray*}
&& |\psi_{j,m}(x)| 
= \Big |  \int_{\mathbb{R}^2}2^{2j}(2^{2j} x_1 + m_1)^{-N_1} \frac{\partial^{N_1}}{\partial \zeta_1^{N_1}} [W(\zeta )] e^{2\pi i \zeta^T (2^{2j} x+m)} d\zeta \Big | \\
& = &  \Big |  \int_{\mathbb{R}^2} 2^{2j}(2^{2j} x_1 + m_1)^{-N_1}(2^{2j} x_2 + m_2)^{-N_2} \frac{\partial^{N_1+N_2}}{\partial \zeta_1^{N_1} \partial \zeta_2^{N_2}} [W(\zeta )] e^{2\pi i \zeta^T (2^{2j} x+m)} d\zeta \Big | \\
& \leq & 2^{2j}| 2^{2j} x_1 + m_1|^{-N_1} |2^{2j} x_2 + m_2|^{-N_2}  \int_{\mathbb{R}^2} \Big | \frac{\partial^{N_1+N_2}}{\partial \zeta_1^{N_1} \partial \zeta_2^{N_2}} [W(\zeta )] \Big | d\zeta ,
\end{eqnarray*}
and similarly
$$
|\psi_{j,m}(x)| 
 \leq  2^{2j}| 2^{2j} x_i + m_i|^{-N_i}  \int_{\mathbb{R}^2} \Big | \frac{\partial^{N_i}}{\partial \zeta_i^{N_i} } [W(\zeta )] \Big | d\zeta,
$$
for $i=1,2$.
Note that the boundary terms vanish since  $W(\zeta)$ has compact support.
Thus, \\
$2^{-2j} \Big ( 1+ |2^{2j} x_1 + m_1|^{N_1}+ |2^{2j} x_2 + m_2|^{N_2} + |2^{2j} x_1 +  m_1|^{N_1}|2^{2j} x_2 + m_2|^{N_2} \Big )  |\psi_{j,m}(x)| $
\begin{eqnarray*}
 & = & 2^{-2j}(1+ |2^{2j} x_1 + m_1|^{N_1})(1+ |2^{2j} x_2 + m_2|^{N_2} )  |\psi_{j,m}(x)| \phantom{11111111111111111} \\
& \leq &  \int_{\mathbb{R}^2}  | W(\zeta )   | d\zeta +  \int_{\mathbb{R}^2} \Big | \frac{\partial^{N_1}}{\partial \zeta_1^{N_1}} W(\zeta ) \Big | d\zeta +  \int_{\mathbb{R}^2} \Big | \frac{\partial^{N_2}}{\partial \zeta_2^{N_2}} W(\zeta ) \Big | d\zeta  \\
&& +  \int_{\mathbb{R}^2} \Big | \frac{\partial^{N_1+N_2}}{\partial \zeta_1^{N_1} \partial \zeta_2^{N_2}} W(\zeta ) \Big | d\zeta .
\end{eqnarray*}
Since $W(\zeta) \in C^{\infty}$\leave{and has compact support}, there exists a constant $C_{N_1, N_2}^\prime$ independent of $j$ such that \\
 $ \int_{\mathbb{R}^2} |W(\zeta ) | d\zeta + \int_{\mathbb{R}^2}  \Big | \frac{\partial^{N_1}}{\partial \zeta_1^{N_1}} W(\zeta )  \Big | d\zeta+ \int_{\mathbb{R}^2}  \Big | \frac{\partial^{N_2}}{\partial \zeta_2^{N_2}} W(\zeta ) \Big | d\zeta +  \int_{\mathbb{R}^2} \Big | \frac{\partial^{N_1+N_2}}{\partial \zeta_1^{N_1} \partial \zeta_2^{N_2}} W(\zeta ) \Big | d\zeta  \leq C_{N_1, N_2}^\prime.$
Thus, we obtain
$$ |\psi_{j,m}(x)| \leq C_{N_1,N_2}^\prime \cdot 2^{2j} \cdot \frac{1}{1+ |2^{2j} x_1 + m_1|^{N_1}}\frac{1}{1+ |2^{2j} x_2 + m_2|^{N_2}}.$$
In addition,  there exist \alex{constants} $C_{N_1}^\prime, C_{N_2}^\prime $ for each $N_1, N_2=1,2, \dots $ such that 
$$\langle | 2^{2j} x_1 + m_1 | \rangle^{N_1}= (1+ | 2^{2j} x_1 + m_1 |^2 )^{N_1/2 } \leq C_{N_1}^\prime (1 + | 2^{2j} x_1 + m_1 |^{N_1}),$$
and 
$$\langle | 2^{2j} x_2 + m_2 | \rangle^{N_2}= (1+ | 2^{2j} x_2 + m_2 |^2 )^{N_2/2 } \leq C_{N_2}^\prime (1 + | 2^{2j} x_2 + m_2 |^{N_2}).$$
Finally, we obtain
$$ |\psi_{j,m}(x)| \leq C_{N_1,N_2} \cdot 2^{2j} \cdot \langle |2^{2j} x_1 + m_1|\rangle^{-N_1}\langle |2^{2j} x_2 + m_2|\rangle^{-N_2}.$$
\leave{This concludes the claim.}
\alex{This shows the first claim.}\\
ii) \leave{Applying the change of}\alex{With the} variable $\zeta^T =\xi^T A_{\alpha, \rm{\iota}}^{-j} S^{-l}_{\rm{\iota}},$ we obtain
\begin{eqnarray*}
&& | \check{\chi}_\iota(x)*\phi_{j, l, k}^{\alpha, \rm{\iota}}(x)| = \Big | \int_{\mathbb{R}^2}\chi_\iota(\xi) \hat{\phi}_{j,l,k}^{\alpha, \rm{\iota}}(\xi) e^{2\pi i \xi^T x} d\xi\Big | \phantom{11111111111111111111}\\
&= & \Big | \int_{\mathbb{R}^2}  2^{-(2+\alpha)j/2}\chi_\iota(\xi)W_j(\xi) V_{\rm{\iota}} \Big ( \xi^T A_{\alpha, \rm{\iota}}^{-j} S^{-l}_{\rm{\iota}} \Big ) e^{2\pi i \xi^T (x + A_{\alpha, \rm{\iota}}^{-j} S^{-l}_{\rm{\iota}}k )}  d\xi\Big | \\
&= & \Big | \int_{\mathbb{R}^2}  2^{(2+\alpha)j/2}\chi_\iota(\xi)W_j((S^{l}_{\rm{\iota}}A_{\alpha, \rm{\iota}}^{j} )^T \zeta) V_{\rm{\iota}} (\zeta) e^{2\pi i \zeta^T (S^{l}_{\rm{\iota}}A_{\alpha, \rm{\iota}}^{j} x + k )}  d\zeta \Big |. 
\end{eqnarray*}
By a similar approach as in i) the decay \leave{estimate }of each shearlet element is then \alex{estimated}\leave{obtained} by
$$ | \check{\chi}_{\rm{v}}*\phi_{j, l, k}^{\alpha, \rm{v}}(x)| \leq C_N \cdot 2^{(2+\alpha)j/2} \cdot \langle |2^{\alpha j}x_1 - k_1|\rangle^{-N} \langle |2^{2j}x_2+ l2^{\alpha j } x_1-k_2|\rangle^{-N}, $$
$$ | \check{\chi}_{\rm{h}}*\phi_{j, l, k}^{\alpha, \rm{h}}(x)| \leq C_N \cdot 2^{(2+\alpha)j/2} \cdot \langle |2^{2j}x_1+l2^{\alpha j}x_2-k_1)|\rangle^{-N} \langle |2^{\alpha j}x_2-k_2|\rangle^{-N}. $$
The same fact is true with the roles of $\chi_{\rm{\iota}}$ and $\gamma_{\iota}$ interchanged.
It should be mentioned that here we use Lemma \ref{865}, ii) to estimate the constant $C_N$. \\
iii) By using i) and ii) the change of variables $(y_1, y_2)=2^{2j}(x_1, x_2)$ we easily verify the claim. Indeed, for $\iota=\rm{v}$ we have
\begin{eqnarray*}
 |\langle \psi_{j,m}, \check{\chi}_{\rm{v}}*\phi^{\alpha, \rm{v}}_{j, l, k} \rangle| &\leq & \int_{\mathbb{R}^2} |\psi_{j,m}(x)| |\check{\chi}_{\rm{v}}*\phi^{\alpha, \rm{v}}_{j,l,k}(x) |dx\\
 &\leq & C_N  2^{2j} 2^{(2+\alpha)j/2} \int_{\mathbb{R}^2} \langle |2^{2j} x_1 + m_1|\rangle^{-N}\langle |2^{2j} x_2 + m_2|\rangle^{-N}  \\
 && \cdot \langle |2^{\alpha j}x_1 + k_1|\rangle^{-N} \langle |2^{2j}x_2+ l2^{\alpha j } x_1+k_2|\rangle^{-N} dx\\
 &\leq & C_N 2^{-(2-\alpha)j/2}\int_{\mathbb{R}^2}  \langle |y_1 + m_1|\rangle^{-N}\langle |y_2 + m_2|\rangle^{-N} dy_1dy_2 \\
 &\leq & C^{\prime}_N 2^{-(2-\alpha)j/2}.
\end{eqnarray*}
\end{proof}
For the other cases, the proofs
are similar.

Such decay estimates in Lemma \ref{465} are particularly useful for estimating the cluster coherence introduced next as well as the cluster sparsity in Subsection \ref{1306-5}.

\begin{proposition} \label{1606-2}
Consider wavelet frame $\mathbf{\Psi}$ and bandlimited $\alpha$-shearlet  frame $\mathbf{\Phi_\alpha}$, we have $$\mu_c(\Lambda_{1j}, \mathbf{\Psi}; \mathbf{\Phi_\alpha}) \rightarrow 0, \; j \rightarrow +\infty.$$
\end{proposition}
\begin{proof}

By definition, we have
$$\mu_c(\Lambda_{1j}, \mathbf{\Psi}; \mathbf{\Phi_\alpha}) = \max_{(\bar{j}, l,k, \iota) \in \Delta} \sum_{(j^{\prime},m) \in \Lambda_{1,j}} |\langle \psi_{j^{\prime},m}, \check{\chi}_\iota(x)*\phi_{\bar{j},l,k}^{\alpha, \iota}\rangle |.$$
Without loss of generality, we assume that the maximum is attained at $\check{\chi}_{\rm{v}}(x)*\phi_{j, l^{\prime}, k^{\prime}}^{\alpha, \rm{v}}.$ Cases $\iota =\rm{h}, \bar{j}=\{j-1, j+1\} $ are done similarly. By Lemma \ref{465}, iii), there exist a constant $C>0$ such that
\begin{eqnarray}
\mu_c(\Lambda_{1j}, \mathbf{\Psi}; \mathbf{\Phi_\alpha}) &\leq & \# \{ (j,m) \in \Lambda_{1,j} \} \cdot  C \cdot 2^{-(2-\alpha)j/2} \nonumber \\
&\leq &C \cdot 2^{2\epsilon j} \cdot  2^{-(2-\alpha)j/2} \nonumber \\
&=& C \cdot 2^{-(2-\alpha -4 \epsilon)j/2} \xrightarrow[j \rightarrow +\infty ]{} 0 \nonumber,
\end{eqnarray}
where the last estimate is due to $2-\alpha-4 \epsilon >0 $ by \eqref{468}.
This finishes the proof.
\end{proof}
\begin{proposition} \label{1606-3}
Consider wavelet frame $\mathbf{\Psi}=\{ \psi_{j,k}\}_{(j,k)}$ and band limited $\alpha$-shearlet frame $\mathbf{\Phi_\alpha}=\{ \phi_{j,l,k}^{\alpha, \iota}\}_{(j,l,k,\iota) \in \Delta}$, we have $$\mu_c(\Lambda_{2j}, \mathbf{\Phi_\alpha}; \mathbf{\Psi}) \rightarrow 0, \; j \rightarrow +\infty.$$
\end{proposition}
\begin{proof}
By definition, we have
$$\mu_c(\Lambda_{2j}, \mathbf{\Phi_\alpha}; \mathbf{\Psi}) = \max_{\bar{j} \in \mathbb{N}, m \in \mathbb{Z}^2} \sum_{(j^{\prime},l,k,\rm{v}) \in \Lambda_{2,j}} |\langle \check{\chi}_\iota(x)*\phi_{j^{\prime},l,k}^{\alpha, \rm{v}}, \psi_{\bar{j},m}\rangle |.$$
We assume that the maxim\alex{um}\leave{al} is attained at $\psi_{j, m^{\prime}}$. By Lemma \ref{465}, i) and ii), we have
\begin{eqnarray}
&& \mu_c(\Lambda_{2j}, \mathbf{\Phi_\alpha}; \mathbf{\Psi}) \leq  C_N   2^{(6+\alpha)j/2} \sum_{\substack{|l| \leq 1, |j^{\prime}-j| \leq 1 \\k \in \mathbb{Z}^2, |k_2 -  lk_1 | < 2^{\epsilon j^\prime} }} \int_{\mathbb{R}^2} \langle |2^{2j} x_1 + m^{\prime}_1|\rangle^{-N} \cdot \nonumber \\
&& \quad  \langle |2^{2j} x_2 + m^{\prime}_2|\rangle^{-N}  \langle |2^{\alpha j^{\prime}}x_1 + k_1|\rangle^{-N} \langle |2^{2j^{\prime}}x_2+ l2^{\alpha j^\prime } x_1+k_2|\rangle^{-N} dx. \nonumber
\end{eqnarray}
Without loss of generality, we consider only the case $j^{\prime}=j$, the cases $j^{\prime}=j-1, j+1$ follow similarly.

Indeed, by the change of variable $(y_1, y_2) = (2^{2j}x_1, 2^{2j}x_2),$ we obtain
\begin{eqnarray}
\mu_c(\Lambda_{2j}, \mathbf{\Phi_\alpha}; \mathbf{\Psi}) &\leq & C_N   2^{-(2-\alpha)j/2} \sum_{\substack{l \in \{ -1, 0,1\}, k \in \mathbb{Z}^2 \\ |k_2 -  lk_1 | < 2^{\epsilon j^\prime} }}
\int_{\mathbb{R}^2} \langle |y_1 + m^{\prime}_1|\rangle^{-N}\langle |y_2 + m^{\prime}_2|\rangle^{-N}  \nonumber \\
&&  \qquad \cdot \langle |2^{-(2-\alpha) j}y_1 + k_1|\rangle^{-N} \langle |y_2+ l2^{\alpha j } x_1+k_2|\rangle^{-N} dy. \label{466}
\end{eqnarray}
In addition, there exists a constant $C^{\prime}$ such that
\begin{equation}
    \sum_{ k \in \mathbb{Z}^2 }  \langle |2^{-(2-\alpha) j}y_1 + k_1|\rangle^{-N} \langle |y_2+ l2^{\alpha j } x_1+k_2|\rangle^{-N}   \leq C^{\prime}. \label{467}
\end{equation}
Combining \eqref{466} with \eqref{467}, we obtain
\begin{eqnarray}
\mu_c(\Lambda_{2j}, \mathbf{\Phi_\alpha}; \mathbf{\Psi}) &\leq & C^{\prime}_N \cdot   2^{-(2-\alpha)j/2}
\int_{\mathbb{R}^2} \langle |y_1 + m^{\prime}_1|\rangle^{-N}\langle |y_2 + m^{\prime}_2|\rangle^{-N} dy \nonumber \\
&\leq & C_N^{''} \cdot 2^{-(2-\alpha)j/2} \xrightarrow{j \rightarrow +\infty } 0.
\end{eqnarray}
This concludes the claim.

\end{proof}

\subsection{Sparse representation error} \label{1306-5}

\begin{proposition} \label{1606-4}
For all $N=1,2, \dots$, we have $\delta_{1j} = \circ(2^{-Nj}).$
\end{proposition}
\begin{proof}
We remark that it is sufficient to consider  $\mathcal{P}(x)=\frac{1}{|x|^{\lambda}}, \lambda \in (0,2).$ The general case is then done by combining all single cases and triangle inequality. 

We first observe that 
$\hat{\mathcal{P}}(\xi) =\frac{1}{|\xi|^{2-\lambda}} $. Thus,
\begin{eqnarray}
\delta_{1,j} &=& \sum_{(j^{\prime},m) \notin \Lambda_{1,j}}\langle \psi_{j^{\prime},m}, \mathcal{P}_j \rangle  \nonumber \\
&=& \sum_{|j^{\prime}-j| \leq 1, |m| > 2^{\epsilon j} }\int_{\mathbb{R}^2}  2^{-2j}W_{j^{\prime}}(\xi)e^{2 \pi i\xi^T\frac{m}{2^{2j}}}W_j(\xi) \frac{1}{|\xi|^{2-\lambda}}  d\xi. \nonumber 
\end{eqnarray}
Since $W_{j^{\prime}}(\xi)W_j(\xi) =0$ for $|j^{\prime} - j | >1 $ due to their supports. Without loss of generality we can assume that $j^{\prime}=j,$ the other cases are estimated analogously.

By the change of variables $\eta=(\eta_1, \eta_2) =2^{-2j}(\xi_1, \xi_2),$ we obtain
\begin{eqnarray}
\delta_{1,j} 
& \lesssim & 2^{2(\lambda-1)j} \sum_{|m| > 2^{\epsilon j} }\int_{\mathbb{R}^2}  \frac{1}{|\eta|^{2-\lambda}}  W^2(\eta)e^{2 \pi i\eta^Tm} d\eta. \nonumber \\
\end{eqnarray}
 Using integration by parts for $N =1,2, \dots, $  with respect to $\eta_1, \eta_2$, respectively, yields
\begin{eqnarray*}
\delta_{1,j} &\lesssim &  2^{2(\lambda-1)j} 
\sum_{\substack{ m_1, m_2 \neq 0, \\  |m| > 2^{\epsilon j} }} \Big |
\int_{\mathbb{R}^2} |m_1|^{-N} |m_2|^{-N} \frac{\partial^{2N}}{\partial \eta_1^{N}\partial \eta_2^{N}} \Big [\frac{1}{|\eta|^{2-\lambda}}W(\eta ) \Big ] e^{2\pi i \eta^T m} d\eta \Big | \\
& \leq  & C_{N} 2^{2(\lambda-1)j}  \sum_{|m|>2^{\epsilon j}} |m_1|^{-N}|m_2|^{-N} \\
& \leq  & C_{N} 2^{2(\lambda-1)j} \sum_{m_1 \in \mathbb{Z}} \sum_{|m_2|\geq \frac{2^{\epsilon j}}{2}} |m_1|^{-N}|m_2|^{-N}  \\
& \leq  & C_{N} 2^{2(\lambda-1)j} 2^{-(N-1)(\epsilon j-1)},
\end{eqnarray*}
for any arbitrary number $N \in \mathbb{N}$. 
Note that the boundary terms vanish due to the compact support of $W(\zeta)$.
This concludes the claim. 
\end{proof}

\begin{proposition} \label{1606-5}
We have $\delta_{2j} = \circ(2^{-Nj}),$ for $\forall N=1,2, \dots$
\end{proposition}
\begin{proof}
By definition, we have
\begin{eqnarray}
\delta_{2,j} &=& \sum_{(j^{\prime},l,k, \iota) \notin \Lambda_{2,j}} |\langle w\mathcal{L}_{j},\check{\chi}_\iota* \phi^{\alpha,\rm{\iota}}_{j^{\prime},l,k} \rangle | \nonumber \\
&=& \sum_{\substack{\{ (j^{\prime},l,k, \rm{h}) \in \Delta \}  \cup \\   \{(j^{\prime},l,k, \rm{v}) \in \Delta, |l|>1 \} }}
|\langle \widehat{w\mathcal{L}}_{j}, \chi_{\iota} \hat{\phi}_{j^{\prime},l,k}^{\alpha, \rm{\iota}}\rangle | +  \sum_{\substack{(j^{\prime},l,k, \rm{v}) \in \Delta,  \\ |l| \leq 1, |k_2-lk_1|>2^{\epsilon j} }} |\langle w\mathcal{L}_{j},\check{\chi}_{\rm{v}}* \phi^{\alpha,\rm{v}}_{j^{\prime},l,k} \rangle | \nonumber \\
&=&I_1+I_2, \label{100617}
\end{eqnarray}
where $$ I_1= \sum_{\substack{\{ (j^{\prime},l,k, \rm{h}) \in \Delta \}  \cup \\   \{(j^{\prime},l,k, \rm{v}) \in \Delta, |l|>1 \} }}
|\langle \widehat{w\mathcal{L}}_{j}, \chi_{\iota} \hat{\phi}_{j^{\prime},l,k}^{\alpha, \rm{\iota}}\rangle |, I_2=\sum_{\substack{(j^{\prime},l,k, \rm{v}) \in \Delta,  \\ |l| \leq 1, |k_2-lk_1|>2^{\epsilon j} }} |\langle w\mathcal{L}_{j},\check{\chi}_{\rm{v}}* \phi^{\alpha,\rm{v}}_{j^{\prime},l,k} \rangle | .$$
Since $W_j(\xi) W_{j^\prime}(\xi)=0$ for $|j^\prime -j|>1$ due to the different support we only consider  $j^{\prime}=j$ without loss of generality. Before starting with $I_1$, we put
\begin{equation} \label{965}
    t^{\rm{v}} = (t_1^{\rm{v}}, t_2^{\rm{v}}): = A^{-j}_{\alpha, \rm{v}} S^{-l}_{\rm{v}} k= (2^{-\alpha j}k_1, 2^{-2 j} (k_2-lk_1)), 
\end{equation} 
\begin{equation}  \label{966}
  t^{\rm{h}} = (t_1^{\rm{h}}, t_2^{\rm{h}}): = A^{-j}_{\alpha, \rm{h}} S^{-l}_{\rm{h}} k= (2^{-2j}(k_1-lk_2), 2^{-\alpha j} k_2).  
\end{equation}

We have
\begin{eqnarray*}
\Big | \langle w\mathcal{L}_j, \check{\chi}_{\iota}*\phi_{j,l,k}^{\alpha, \rm{\iota}} \rangle \Big | & =& \Big | \langle \widehat{w\mathcal{L}}_j, \chi_\iota \hat{\phi}_{j,l,k}^{\alpha, \rm{\iota}} \rangle \Big | = \Big | \int_{\mathbb{R}^2} \hat{w}(\xi_1) \chi_\iota(\xi) W_j(\xi) \overline{\hat{\phi}^{\alpha, \rm{\iota}}_{j,l,k}(\xi)} \Big |\\
& = & \Big | \int_\mathbb{R}e^{2\pi i t_2^{\rm{\iota}}\xi_2} \int_\mathbb{R}\hat{w}(\xi_1)\chi_\iota(\xi)W_j(\xi)\hat{\phi}^{\alpha, \rm{\iota}}_{j,l,0}(\xi)e^{2\pi it_1^{\rm{\iota}}\xi_1}d\xi_1 d\xi_2 \Big | \\
&=& \Big |  \int_\mathbb{R}e^{2\pi i t_2^{\rm{\iota}}\xi_2} \int_\mathbb{R} \hat{w}(\xi_1) \Theta^{\iota}_{j,l}(\xi)e^{2\pi it_1^{\rm{\iota}}\xi_1}d\xi_1 d\xi_2 \Big |,
\end{eqnarray*}
where $\Theta^{\iota}_{j,l}(\xi):=  \chi_\iota(\xi) W_j(\xi)\hat{\phi}^{\alpha, \rm{\iota}}_{j,l,0}(\xi) $.

Applying repeated integration by parts twice with respect to $\xi_i, i=1,2$, we have
\begin{eqnarray}
  \Big | \langle w\mathcal{L}_j, \phi_{j,l,k}^{\alpha, \rm{\iota}} \rangle \Big |  &=& \Big |  \int_\mathbb{R}e^{2\pi i t_2^{\rm{\iota}}\xi_2} \Big [ \int_\mathbb{R} \frac{\partial^2}{\partial \xi_1^2} [\hat{w}(\xi_1) \Theta^{\iota}_{j,l}(\xi)]\frac{1}{(2\pi i t_1^{\rm{\iota}})^2}e^{2\pi it_1^{\rm{\iota}}\xi_1}d\xi_1 \Big ] d\xi_2 \Big |  \nonumber \\
  & \leq & \frac{C}{|t_1^{\rm{\iota}}|^2} \Big | \int_\mathbb{R}  e^{2\pi i t_2^{\rm{\iota}}\xi_2} \int_\mathbb{R} \frac{\partial^2}{\partial \xi_1^2} [\hat{w}(\xi_1) \Theta^{\iota}_{j,l}(\xi)]e^{2\pi it_1^{\rm{\iota}}\xi_1} d\xi_1 d\xi_2 \Big | \nonumber \\ 
 & \leq & C|t_1^{\rm{\iota}}|^{-2} \int_\mathbb{R} \int_\mathbb{R} \Big | \frac{\partial^2}{\partial \xi_1^2} [\hat{w}(\xi_1) \Theta^{\iota}_{j,l}(\xi)] \Big | d\xi_1 d \xi_2.  \label{9670} 
\end{eqnarray}
Similarly, we obtain
\begin{eqnarray}
 \Big | \langle w\mathcal{L}_j, \phi_{j,l,k}^{\alpha, \rm{\iota}} \rangle \Big | & \leq &  C|t_2^{\rm{\iota}}|^{-2} \int_\mathbb{R} \int_\mathbb{R} \Big | \frac{\partial^2}{\partial \xi_2^2} [\hat{w}(\xi_1) \Theta^{\iota}_{j,l}(\xi)] \Big | d\xi_1 d \xi_2   \label{9671}
 \end{eqnarray}
\begin{eqnarray}
 \Big | \langle w\mathcal{L}_j, \phi_{j,l,k}^{\alpha, \rm{\iota}} \rangle \Big | & \leq & C |t_1^{\rm{\iota}}|^{-2} |t_2^{\rm{\iota}}|^{-2} \int_\mathbb{R} \int_\mathbb{R} \Big | \frac{\partial^{4}}{\partial \xi_1^2\partial \xi_2^2} [\hat{w}(\xi_1) \Theta^{\iota}_{j,l}(\xi)] \Big | d\xi_1 d \xi_2.  \label{9672}
\end{eqnarray}
Therefore, we have
$$
 |t_1^{\rm{\iota}}|^2  \Big | \langle w\mathcal{L}_j, \phi_{j,l,k}^{\alpha, \rm{\iota}} \rangle \Big | \quad \stackrel{\mathclap{\normalfont{\eqref{9670}}}}{ \leq } \quad  C \int_\mathbb{R} \int_\mathbb{R} \Big | \frac{\partial^2}{\partial \xi_1^2} [\hat{w}(\xi_1) \Theta^{\iota}_{j,l}(\xi)] \Big | d\xi_1 d \xi_2  $$
 
 $$ |t_2^{\rm{\iota}}|^2  \Big | \langle w\mathcal{L}_j, \phi_{j,l,k}^{\alpha, \rm{\iota}} \rangle \Big |  \quad \stackrel{\mathclap{\normalfont{\eqref{9671}}}}{ \leq } \quad  C \int_\mathbb{R} \int_\mathbb{R} \Big | \frac{\partial^2}{\partial \xi_2^2} [\hat{w}(\xi_1) \Theta^{\iota}_{j,l}(\xi)] \Big | d\xi_1 d \xi_2  $$
 
  $$|t_1^{\rm{\iota}}|^2 |t_2^{\rm{\iota}}|^2  \Big | \langle w\mathcal{L}_j, \phi_{j,l,k}^{\alpha, \rm{\iota}} \rangle \Big | \quad \stackrel{\mathclap{\normalfont{\eqref{9672}}}}{ \leq }  \quad  C\int_\mathbb{R} \int_\mathbb{R} \Big | \frac{\partial^{4}}{\partial \xi_1^2\partial \xi_2^2} [\hat{w}(\xi_1) \Theta^{\iota}_{j,l}(\xi)] \Big | d\xi_1 d \xi_2. $$
  These imply
  $$ (1+ |t_1^{\rm{\iota}}|^2 + |t_2^{\rm{\iota}}|^2+ |t_1^{\rm{\iota}}|^2 |t_2^{\rm{\iota}}|^2)  \Big | \langle w\mathcal{L}_j, \phi_{j,l,k}^{\alpha, \rm{\iota}} \rangle \Big | =   \langle  |t_1^{\rm{\iota}}| \rangle^2  \langle | t_2^{\rm{\iota}}| \rangle^2 \Big | \langle w\mathcal{L}_j, \phi_{j,l,k}^{\alpha, \rm{\iota}} \rangle \Big |  $$
 \begin{eqnarray}
    &\leq&  C\int_{\mathbb{R}^2} \Big |  \hat{w}(\xi_1) \Theta^{\iota}_{j,l}(\xi) \Big |+ \Big | \frac{\partial^2}{\partial \xi_1^2} [\hat{w}(\xi_1) \Theta^{\iota}_{j,l}(\xi)] \Big |+\Big | \frac{\partial^2}{\partial \xi_2^2} [\hat{w}(\xi_1) \Theta^{\iota}_{j,l}(\xi)] \Big | \qquad  \qquad \qquad \nonumber \\ && + \Big | \frac{\partial^{4}}{\partial \xi_1^2\partial \xi_2^2} [\hat{w}(\xi_1) \Theta^{\iota}_{j,l}(\xi)] \Big | d\xi =C\int_{\mathbb{R}} \Gamma(\xi_2) d\xi_2, \nonumber
 \end{eqnarray}
where  
$\Gamma(\xi_2):= 
\int_{\mathbb{R}} \Big |  \hat{w}(\xi_1) \Theta^{\iota}_{j,l}(\xi) \Big |+ \Big | \frac{\partial^2}{\partial \xi_1^2} [\hat{w}(\xi_1) \Theta^{\iota}_{j,l}(\xi)] \Big |+\Big | \frac{\partial^2}{\partial \xi_2^2} [\hat{w}(\xi_1) \Theta^{\iota}_{j,l}(\xi)] \Big |+ \Big | \frac{\partial^{4}}{\partial \xi_1^2\partial \xi_2^2} [\hat{w}(\xi_1) \Theta^{\iota}_{j,l}(\xi)] \Big | d\xi_1.$

Now we investigate the term $\Gamma$ further.
By the definition, we have $\hat{\psi}_{j,l,k}^{\alpha, \rm{\iota}} $ has compact support in the trapezoidal regions by \eqref{96100}, \eqref{96101}. This implies that for any $\xi \in \supp \Theta^{\iota}_{j,l},$ with $\iota =\rm{v}, $ or $\iota=\rm{h}, |l|>1,$ there exist constants $C_1, C_2 >0$ such that
$$ \xi_1 \in I_{j,l}:= [-C_2l2^{\alpha j },-C_1l2^{\alpha j }] \cup  [C_2l2^{\alpha j },C_1l2^{\alpha }].$$
Thus, we obtain the decay estimate of $\hat{w}^{(n)}(\xi_1), \xi_1 \in I_{j,l}, n=0,1,2,$ by
\begin{equation}
    |\hat{w}^{(n)}(\xi_1)| \leq C_N \langle | 2^{\alpha j}| \rangle^{-N}, \; \forall N \in \mathbb{N}_0. \label{10061}
\end{equation}
For an illustration, we refer to Fig. \ref{2108-1}
\begin{figure}[H] 
  \centering
  \includegraphics[width=0.50\textwidth]{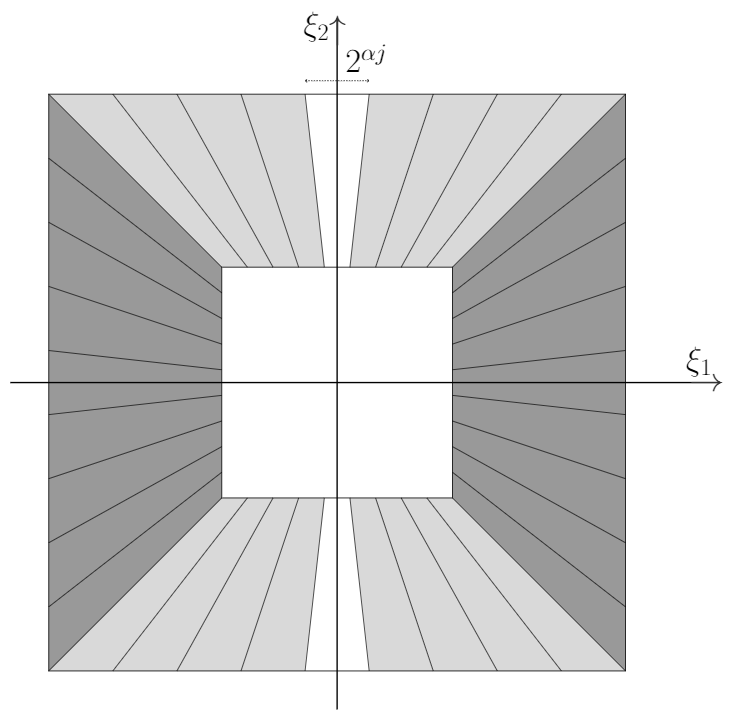}

  \caption{\label{2108-1} Frequency support of  shearlets with $\iota = \rm{v}, |l|>1$ (gray) and  shearlets with $\iota=\rm{h} $ (black).}
\end{figure}

By a direct computation, we also have ${\Theta^{\iota}_{j,l}}^{(n)} \leq C_n, n=0,1,2.$ Combining this with \eqref{10061}, we obtain
$\int_{\mathbb{R}} \Gamma(\xi_2) d\xi_2 \leq C_N \langle | 2^{\alpha j}| \rangle^{-N}, \; \forall N \in \mathbb{N}_0.$ Thus, 
$$\Big | \langle w\mathcal{L}_j, \phi_{j,l,k}^{\alpha, \rm{\iota}} \rangle \Big | \leq  C_N \langle | 2^{\alpha j}| \rangle^{-N}  \langle  |t_1^{\rm{\iota}}| \rangle^2  \langle | t_2^{\rm{\iota}}| \rangle^2 .$$ This implies
\begin{eqnarray}
   I_1 &\leq & C_N \langle | 2^{\alpha j}| \rangle^{-N} \sum_{\substack{(j,l,k, \rm{h}) \cup  \\  \{\iota =\rm{v}, |l|>1 \} }}  \langle  |t_1^{\rm{\iota}}| \rangle^{-2}  \langle | t_2^{\rm{\iota}}| \rangle^{-2}. \label{10063}
\end{eqnarray}
In addition, we have
\begin{eqnarray}
   \sum_{k_1, k_2 \in \mathbb{Z}}  \langle  |t_1^{\rm{h}}| \rangle^{-2}  \langle | t_2^{\rm{h}}| \rangle^{-2} &=&  \sum_{k_1, k_2 \in \mathbb{Z}}  \langle  |2^{-2j}(k_1-lk_2)| \rangle^{-2}  \langle | 2^{-\alpha j} k_2| \rangle^2\nonumber \\
   &=&  \sum_{k_1, k_2 \in \mathbb{Z}}  \langle  |2^{-2j}k_1| \rangle^{-2}  \langle | 2^{-\alpha j} k_2| \rangle^{-2}\nonumber \\
   &\leq &  \int_{\mathbb{R}} \langle  |2^{-2j}t_1| \rangle^{-2}  dt_1  \int_{\mathbb{R}}\langle | 2^{-\alpha j} t_2| \rangle^{-2} dt_2\nonumber \\
    &\leq & C 2^{(2+\alpha)2j}. \label{10064}
    \end{eqnarray}
Similarly, we obtain  
\begin{equation}
    \sum_{k_1, k_2 \in \mathbb{Z}}  \langle  |t_1^{\rm{v}}| \rangle^{-2}  \langle | t_2^{\rm{v}}| \rangle^{-2}  \leq  C 2^{(2+\alpha)2j}. \label{10065}
\end{equation}
In addition, we have 
\begin{equation}
    |l| \leq \lceil 2 \cdot 2^{(2-\alpha )j} \rceil.  \label{10066}
\end{equation}
Combining \eqref{10063}, \eqref{10064}, \eqref{10065}, and \eqref{10066} we obtain
\begin{equation}
    I_1 \leq C_N  \lceil 2 \cdot 2^{(2-\alpha )j} \rceil \langle | 2^{\alpha j}| \rangle^{-N}. \label{10067}
\end{equation}
We now  turn our attention to the term $I_2$. We have
\begin{eqnarray}
| w \mathcal{L}_j (x) | &=& \Big | \int_\mathbb{R} w(y_1) F_j(x-(y_1, 0)) dy_1 \Big | \phantom{111111111111111111111111111}\nonumber \\
& \leq & \Big | \int_\mathbb{R} |w(y_1) | \cdot 2^{4j}| \check{W}( 2^{2j}(x-(y_1, 0))) | dy_1 \Big | \nonumber \\
& \leq &   \int_\mathbb{R} |w(y_1)  | \cdot  C_N 2^{4j}  \langle | 2^{2j}x_2|\rangle^{-N}  \langle | 2^{2j}(y_1 -x_1)|\rangle^{-N}  dy_1  \nonumber \\
&=& C_N2^{4j} \langle | 2^{2j} x_2 | \rangle^{-N} [ |w * \langle 2^{2j}[\cdot]| \rangle^{-N} ](x_1) \nonumber \\
 &=&  C_N2^{4j} \langle | 2^{2j} x_2 | \rangle^{-N}  \tau_{N, j}(x_1),  \label{10069}
\end{eqnarray}
where $\tau_{N, j}(x_1): =   [ |w * \langle 2^{2j}[\cdot]| \rangle^{-N} ](x_1), \quad x=(x_1, x_2) \in \mathbb{R}^2.$ 
Combining \eqref{10069} with Lemma \ref{465}, we obtain
\begin{eqnarray}
I_2 &\leq &  C_N  \sum_{\substack{l \in \{-1,0,1 \},\\  |k_2 -lk_1| > 2^{\epsilon j} } } \int_{\mathbb{R}^2} 2^{4j} \langle | 2^{2j}x_2 |\rangle^{-N} \tau_{N,j}(x_1) 2^{(\alpha +2 )j/2} \langle |2^{\alpha j} x_1 - k_1 | \rangle^{-N} \nonumber \\ 
&& \qquad    \cdot \langle | l 2^{\alpha j }x_1 + 2^{2j} x_2 - k_2 | \rangle^{-N} dx \nonumber \\
& = &  C_N 2^{(6-\alpha ) j/2} \sum_{\substack{l \in \{-1,0,1 \},\\  |k_2 -lk_1| > 2^{\epsilon j} } } \int_{\mathbb{R}^2}  \langle |x_2 |\rangle^{-N} \tau_{N,j}(2^{-\alpha j}(x_1 + k_1)) \langle |x_1| \rangle^{-N} \nonumber \\
&& \qquad \quad \qquad \quad \quad \cdot \langle | lx_1 + lk_1+ x_2 - k_2 | \rangle^{-N} dx \nonumber \\
&  \leq & C_N 2^{(6-\alpha)j/2} \sum_{\substack{l \in \{-1,0,1 \},\\  |k_2 -lk_1| > 2^{\epsilon j} } } \int_{\mathbb{R}} \tau_{N,j}(2^{-\alpha j}(x_1 + k_1)) \langle |x_1| \rangle^{-N} \cdot \nonumber \\
&& \langle | lx_1 + lk_1 -k_2 | \rangle^{-N} dx_1. \nonumber 
\end{eqnarray}
The last equality comes from the fact that  there exists a constant $C_N>0$ satisfying
\begin{equation}
    \int_{\mathbb{R}} \langle |x| \rangle^{-N} \langle | x + a| \rangle^{-N} dx
\leq  C_N \langle |a| \rangle^{-N}, \quad  \forall a \in \mathbb{R}. \label{100615}
\end{equation}
In addition, we have
\begin{eqnarray*}
\sum_{k_1 \in \mathbb{Z}}\tau_{N,j}(2^{-\alpha j}(x_1 + k_1)) & = & \sum_{k_1 \in \mathbb{Z}} \int_{\mathbb{R}}|w(y_1) | \langle | 2^{2j}(y_1-2^{-\alpha j }(x_1+ k_1))| \rangle^{-N}dy_1 \\
& = & \sum_{k_1 \in \mathbb{Z}} \int_{\mathbb{R}} |w(y_1) | \langle | 2^{(2-\alpha)j}(k_1 + x_1 - 2^{\alpha j }y_1)| \rangle^{-N}dy_1 \\
& \stackrel{\mathclap{\normalfont{\alpha \leq 2 }}} \leq & \; \int_{\mathbb{R}}  |w(y_1) | \Big (  \sum_{k_1 \in \mathbb{Z}}\langle |k_1 + x_1 - 2^{\alpha j }y_1| \rangle^{-N} \Big ) dy_1 \\
& \leq &  C\int_{\mathbb{R}}  |w(y_1) | \Big (  \int_\mathbb{R} \langle |t + x_1 - 2^{\alpha j }y_1| \rangle^{-N} dt \Big ) dy_1 \\
& = &  C\int_{\mathbb{R}}  |w(y_1) | \Big (  \int_\mathbb{R} \langle |t| \rangle^{-N} dt \Big ) dy_1  \leq C_N.
\end{eqnarray*}
Therefore, we obtain
\begin{eqnarray}
I_2 & \leq &  C_N 2^{(6-\alpha)j/2} \sum_{\substack{ l \in \{ -1,0,1 \}, k_2 \in \mathbb{Z} \\ |k_2 | > 2^{\epsilon j} }} \int_\mathbb{R} \langle | x_1 | \rangle^{-N} \langle |lx_1 + k_2 | \rangle^{-N} dx_1 \nonumber \\
& \stackrel{\mathclap{\normalfont{\textup{by } \eqref{100615}}}} \leq &   \quad C_N 2^{(-\alpha)j/2} \sum_{k_2 \in \mathbb{Z},  |k_2 | > 2^{\epsilon j} }  \langle | k_2 | \rangle^{-N}   dx_1 \nonumber \\
& \leq &  C_N 2^{(6-\alpha)j/2} \int_{t > 2^{\epsilon j}} \langle | t | \rangle^{-N} \leq  C_N 2^{(6-\alpha)j/2} 2^{-(N-1)\epsilon j}, \quad \forall N \in \mathbb{N}_0. \quad \label{100616}
\end{eqnarray}
Finally, combining
 \eqref{10067}, \eqref{100616} with \eqref{100617} concludes the claim. 

\end{proof}

\subsection{Proof of main theorem}
Now \alex{we are well prepared}\leave{it is ready} to provide a proof for Theorem \ref{1006-18}. 
\begin{proof}
Applying Theorem \ref{1206-5} and Propositions \ref{1606-2}, \ref{1606-3}, \ref{1606-4}, \ref{1606-5}, we obtain
\begin{equation} \label{1908-1}
\| \mathcal{P}_j^\star - \mathcal{P}_j  \|_2  + \| \mathcal{C}_j^\star - w\mathcal{L}_j  \|_2  = o(2^{-Nj}) \rightarrow 0, \quad j \rightarrow \infty. 
\end{equation}
Since \eqref{1908-1} holds for arbitrarily large number $N \in \mathbb{N} $ we \alex{are left}\leave{remain} to estimate the lower bounds of $\|\mathcal{P}_j \|_2$ and $\| w\mathcal{L}_j \|_2$. Without loss of generality, we assume that $\mathcal{P}=\frac{1}{|x|^\lambda}, 0< \lambda <2$. The more general case follows by combining translation invariance with many uses of the triangle inequality.
By Plancherel's theorem and $\widehat{(\frac{1}{|x|^{\lambda}})} = \frac{1}{|\xi|^{2-\lambda}}$, we have

\begin{eqnarray*}
\|\mathcal{P}_j \|_2^2 =\|F_j* \frac{1}{|x|^{\lambda}}\|^2_2 &=&\int_{\mathcal{A}_j} W^2_j(\xi) \frac{1}{|\xi|^{2(2-\lambda)}} d\xi \\
& \gtrsim &  c \cdot 2^{(2-2\lambda)2j}.
\end{eqnarray*}
In addition, 
\begin{eqnarray*}
    \| w \mathcal{L}_j \|^2_2 &=& \int_{\mathcal{A}_j} \hat{w}^2(\xi_1 ) W^2_j(\xi) d \xi \\ 
    &\gtrsim& \int_{\xi_1 \in \mathbb{R}} \hat{w}^2(\xi_1 )  d \xi_1 \int_{2^{2j-5}}^{2^{2j-2}} C d \xi_2\\
    &\gtrsim&  C 2^{2j}.
\end{eqnarray*}
This finishes the proof.
\end{proof}

\section{Conclusions and future directions } \label{2906-1}
\textcolor{black}{ Our main results, Theorem \ref{1206-5} and Theorem \ref{1006-18}, 
show that the ground truth components are \alex{under reasonable assumptions} perfectly recovered at fine scales by Algorithm 1. For \alex{a} practical multiscale approach, we first apply the proposed algorithm for each filtered subband image at each scale \alex{and then recover} the whole signal \alex{using the}\leave{is then recovered by using} reconstruction formula \eqref{2606-1}.
These theorems \alex{might be extended}\leave{are possible to extend} in several ways. Here we focus on separating pointlike and curvelike structures, but our  approach holds for other types of geometric components as well as any frames providing sparse representation.  We can therefore use wavelets, curvelets, Gabor frames, as well as other sparsifying systems for representing components \alex{to be separated}\leave{which we want to separate}. A similar asymptotic separation result can \alex{be} achieved by using such sparsifying systems \alex{if they}\leave{as well as they} provide sufficient small cluster sparsity and cluster coherence. }

\textcolor{black}{
 In our future work,  we  will  extend  our  theory  to  the  case  of multiple-component separation. \alex{Our results might also be applicable to images that are corrupted by noise.}
 \leave{Also, images often contain noise, noisy images where the noise is assumed to be sufficient small can be applicable in this direction.}  
 In addition, other algorithms such as thresholding methods \cite{28} or $l_1$-minimization with an additional regularization term \cite{20} which is effective for image separation in empirical work can be considered in our future study.}


%
%


\begin{thebibliography}{99} 

\bibitem{1}
S. Mallat, \textit{A Wavelet Tour of Signal Processing,} CA, San Diego:Academic, 1998.

\bibitem{2}D. L. Donoho and G. Kutyniok, \textit{Geometric Separation using a Wavelet-Shearlet Dictionary}. In L. Fesquet and B. Torrésani, editors, Proceedings of 8th International Conference on Sampling Theory and Applications (SampTA), Marseille, 2009.

\bibitem{3} O. Christensen, {\it An introduction to frames and Riesz bases,} Applied  and Numerical Harmonic  Analysis, Birkh\"auser Boston Inc., Boston, MA, 2003. MR 1946982(2003k:42001).

\bibitem{4}
E. Candes, {\it Ridgelets: theory and applications,} Ph.D. thesis, Department of Statistics,
Stanford University, 1998.


\bibitem{5}
E. Candes, D. Donoho, Curvelets, {\it a surprisingly effective nonadaptive representation for objects with edges,} in: Curves and Surfaces, Vanderbilt
University Press, 1999.

\bibitem{6}
K. Guo, D. Labate,
 {\it Optimally sparse multidimensional representation using shearlets, }
SIAM J. Math. Anal., 39 (2007), pp. 298-318

\bibitem{7}
VT.Do, R. Levie, G. Kutyniok, {\it Analysis of simultaneous inpainting and geometric separation based on sparse decomposition }, Analysis and Applications, doi: 10.1142/S021953052150007X.

\bibitem{8}
V.M. Patel, G. Easley, D. M. Healy,  {\it Shearlet-based deconvolution,} IEEE Trans. Image Process.
18(12) (2009), 2673-2685

\bibitem{9}
P. Kittipoom, G. Kutyniok, W.-Q. Lim, {\it Construction of compactly supported shearlet frames,} Constructive
Approximation 35 (2012) 2962–3017.

\bibitem{10} E.J. King, G. Kutyniok, and X. Zhuang, \textit{Analysis of data separation and recovery problems using clustered sparsity},  Proc. SPIE, Wavelets and Sparsity XIV, vol. 8138, (2011), pp. 813818-813818-11.

\bibitem{11} K. Guo, G. Kutyniok, and D. Labate, {\it Sparse multidimensional representations using anisotropic dilation and shear operators,} Wavelets and Splines (Athens, GA, 2005), Nashboro Press.14 (2006), 189-201.

\bibitem{12} D.L. Donoho, G. Kutyniok, \textit{Microlocal analysis of the geometric separation problem}, Comm. Pure Appl. Math. 66 (2013), no. 1, 1-47. MR 2994548.

\bibitem{13}
G. Kutyniok and W.-Q. Lim, \textit{Compactly supported shearlets are optimally sparse}, J. Approx.
Theory, 163:1564–1589, 2010.


\bibitem{14}
E. LePennec, S. Mallat, \textit{ Sparse geometric image representation with bandelets,} IEEE Transactions on Image Processing 14 (2005) 423–438.


\bibitem{15} E.J.King, G. Kutyniok, X. Zhuang, \textit{Ananasis of Inpainting via clustered sparsity and microlocal analysis,} J.Math. Imaging Vis. 48 (2014), 205-234. 

\bibitem{18} R.J. Duffin, A.C. Schaeffer, \textit{
A class of nonharmonic Fourier series,}
Trans. Amer. Math. Soc., 72 (1952), pp. 341-366.

\bibitem{19}  K. Guo and D. Labate, \textit{ The construction of smooth Parseval frames of shearlets}, Math. Model. Nat. Phenom., 8 (2013), pp. 82–105.
 
\bibitem{20} G. Kutyniok and W.-Q Lim, \textit{Image separation using wavlets and shearlets}. In Curves and Surfaces (Avignon, France, 2010), Lecture Notes in Computer Science 6920, Springer, 2012.

\bibitem{21} P. Grohs, \textit{Bandlimited shearlet frames with nice duals,} J. Comput. Appl. Math. 243, 139–151 (2013).

\bibitem{22} 
P.G. Casazza, 
\textit{The art of frame theory,}
Taiwanese J. Math., 4 (2000), pp. 129-202

\bibitem{23} G. Kutyniok, \textit{Clustered sparsity and separation of cartoon and texture}, SIAM J. Imaging Sci. 6 (2013), 848-874.

\bibitem{24}
G. Kutyniok and W.-Q Lim, \textit{ Dualizable shearlet frames and sparse approximation,} Constr. Approx., 44 (2016), pp. 53--86.

\bibitem{25}
G. Kutyniok and D. Labate, \textit{Introduction to shearlets}, in Shearlets: Multiscale Analysis for Multivariate Data, Cambridge, MA, USA:Birkhäuser, pp. 1-38, 2012.

\bibitem{26}
F. G. Meyer, A. Averbuch, and R. R. Coifman, {\it Multi-layered Image Representation: Application to Image Compression,} IEEE Trans. on
Image Processing 11(9) (2002), 1072–1080.

\bibitem{27}
M. Do, M. Vetterli,  {\it The contourlet transform: an efficient directional multiresolution image representation,} IEEE Transactions on Image Processing 14
(2005) 2091–2106.

\bibitem{28} Kutyniok, G, {\it Geometric separation by single-pass alternating thresholding.} Appl. Comput. Harmon. Anal. 36, 23–50 (2014).

\bibitem{29}
D. Donoho,  {\it Compressed sensing}, IEEE Trans. Inform. Theory, vol. 52, no. 4, pp. 1289-1306, Apr. 2006.

\bibitem{30}
E. Candès, J. Romberg and T. Tao, {\it Stable signal recovery from incomplete and inaccurate measurements,} Comm. Pure Appl. Math., vol. 59, no. 8, pp. 1207-1223, Aug. 2006.

\bibitem{31}
J.-L. Starck, M. Elad and D. L. Donoho, {\it Image decomposition: Separation of texture from piecewise smooth content }, SPIE Meeting, 2003-Aug.

\bibitem{32}
Guo, Kanghui and Labate, Demetrio and Ayllon, Jose, {\it Image inpainting using sparse multiscale representations: image recovery performance guarantees.} Appl. Comput. Harmon. Anal., (2020).

\bibitem{33}
J. Bobin, J. Starck and R. Ottensamer, {\it Compressed Sensing in Astronomy,} in IEEE Journal of Selected Topics in Signal Processing, vol. 2, no. 5, pp. 718-726, Oct. 2008.

\bibitem{34}
D. L. Donoho, M. Vetterli, R. DeVore, and I. Daubechies, {\it Data compression and harmonic
analysis, } IEEE Trans. Info. Theory 44 (1998), 2435–2476.

\bibitem{35}
 E. Cand`es, J. Romberg, and T. Tao. {\it Robust uncertainty principles: Exact
signal reconstruction from highly incomplete frequency information, }
IEEE Transations on Information Theory, vol. 52, pp. 489–509, 2006.

\bibitem{36}
M. Rostami, O. Michailovich and Z. Wang, {\it Image Deblurring Using Derivative Compressed Sensing for Optical Imaging Application,} in IEEE Transactions on Image Processing, vol. 21, no. 7, pp. 3139-3149, July 2012.

\bibitem{37}
G. Kutyniok and D. Labate, {\it Resolution of the wavefront set using continuous shearlets},  Trans. Amer. Math. Soc., vol. 361, pp. 2719-2754, 2009.

\bibitem{39} M. Genzel, G. Kutyniok, \textit{Asymptotic Analysis of Inpainting via Universal Shearlet Systems}, SIAM J. Imaging Sci. 7.4 (2014), 2301-2339. 

\bibitem{40}
M. Elad, J. L. Starck, P. Querre, and D. L. Donoho, {\it Simultaneous cartoon and texture image inpainting using morphological component analysis (MCA)}, Appl. Comput. Harmon. Anal.19(3) (2005), 340–358.

\bibitem{41}
D. L. Donoho, M. Vetterli, R. DeVore, and I. Daubechies, {\it Data compression and harmonic 
analysis,} IEEE Trans. Info. Theory 44 (1998), 2435–2476.

\bibitem{42}
S. Keiper, {\it A flexible shearlet transform – sparse approximation and dictionary learning.} Bachelor’s
thesis, TU Berlin, Germany, 2012.

\bibitem{43}
G. Kutyniok, J. Lemvig, and W.-Q Lim, {\it Optimally sparse approximations of 3D functions by
compactly supported shearlet frames.} SIAM J. Math. Anal., 44(4):2962–3017, 2012.

\bibitem{44} J. Bobin, J.-L. Starck, J. M. Fadili, Y. Moudden, and D. L. Donoho, {\it Morphological component analysis: an adaptive thresholding strategy,} IEEE Trans. Image Process. 16(11), pp. 2675–2681, 2007.

\bibitem{45} J.-L. Starck, M. Elad, and D. L. Donoho, {\it Image decomposition via the combination of sparse representations and a variational approach,} IEEE Trans. Image Process. 14(10), pp. 1570–1582, 2005.

 
\end{thebibliography}


\end{document}